\documentclass[preprint,12pt]{elsarticle}

\usepackage{enumerate}
\usepackage{amssymb}
\usepackage{lipsum}
\usepackage[a4paper, total={7.0in, 9.8in}]{geometry}
\usepackage{booktabs}
\usepackage{amsmath,amssymb,url}

\usepackage{bm}

\usepackage{amsmath}
\usepackage{enumitem} %%% Use the macros provided by this package when list labels stick out into the left margin. See the list in Section 7.
\usepackage{graphics} %%% (Optional. Use for incorporating separate graphics files.)
\usepackage{array}
\usepackage{dsfont}
\usepackage[all]{xy}
\usepackage{amsthm}
\usepackage{tikz}
\usepackage{bbding}
\usepackage{graphicx}

\usepackage{mathrsfs}

\numberwithin{equation}{section}
\newtheorem{theorem}{Theorem}[section]
\newtheorem{lemma}[theorem]{Lemma}
\newtheorem{proposition}[theorem]{Proposition}
\newtheorem{corollary}[theorem]{Corollary}

\newcommand\blfootnote[1]{%
\begingroup
\renewcommand\thefootnote{}\footnote{#1}%
\addtocounter{footnote}{-1}%
\endgroup
}
\journal{}

\begin{document}

\begin{frontmatter}



\title{The existence of the compatible lattice order on the complex field $\mathbb{C}$}


\author{Zhengmao He}
\address{School of Sciences, Southwest Petroleum University,Chengdu 610500,Sichuan China}

\begin{abstract} In this paper, we give an affirmative answer to the open question posed by G.~Birkhoff and R.~S.~Pierce in 1956 concerning the existence of a compatible lattice order on the complex field $\mathbb{C}$. Using valuation rings theory, basis extension in infinite-dimensional vector spaces and extension of field embeddings to algebraic extensions, we obtain that for every real $t>1$ transcendental over $\mathbb{Q}$, the compatible lattice order can be chosen such that $0\prec1$, $0\vee i=t$, and $1\wedge t=0$. Consequently, $\mathbb{C}$ admits at least $2^{\aleph_0}$ distinct compatible lattice orders, each extending the usual order on $\mathbb{Q}$. 
\end{abstract}
\begin{keyword}  compatible lattice order; complex field $\mathbb{C}$; valuation rings theory; field embedding; infinite-dimensional vector space\\
\vspace*{0.2cm}
\end{keyword}

%\subjclass[{\normalfont 2000}]{54B20; 06B35; 06F30}

\end{frontmatter}

\blfootnote{This work is supported by the National Natural Science Foundation of China (Grant nos.12601906) and the Sichuan Science and Technology Program (Grant No. 2026NSFSC0784).}
\blfootnote{E-mail address: hezhengmaomath@163.com (Z.M.He).}


\section{Introduction}
\label{}

The interaction between order and multiplication is a central theme in the theory of ordered algebraic structures. A partial order $\preceq$ on a field $L$ is compatible with its ring operations if translation preserves order and the product of two nonnegative elements is nonnegative. Thus $x\preceq y$ implies $x+a\preceq y+a$, and $0\preceq x$, $0\preceq y$ imply $0\preceq xy$. Such an order is a lattice order if every pair of elements has a least upper bound and a greatest lower bound, denoted by $x\vee y$ and $x\wedge y$. Throughout this paper, a lattice-ordered field means a field equipped with these ring-order and lattice properties. The order need not be total, and positivity of inverses is not included as an additional axiom.

The familiar obstruction to ordering $\mathbb{C}$ concerns total field orders. In a totally ordered field every nonzero square is positive, which is incompatible with $i^2=-1$. For a partial order, however, an element may be incomparable with zero, so this argument does not exclude a compatible lattice order. Conversely, the coordinate order has joins and meets but is incompatible with multiplication: it makes $i$ nonnegative while $i^2=-1$ is negative. These two observations identify the difficulty. One must construct enough order structure to obtain least and greatest bounds while retaining multiplicative closure of the positive cone.

In their 1956 paper on lattice-ordered rings, G.~Birkhoff and R.~S.~Pierce posed the question whether the complex field admits a lattice order compatible with its ring operations~\cite{BP1956}. This question, usually called the Birkhoff--Pierce problem for the complex field, belongs to a broader programme concerning partial orders on familiar fields and division rings. In 1963, Fuchs formulated the related Problem~31, asking for a description of the directed orders on the complex and quaternion fields~\cite[p.~212]{Fuchs1963}. Directedness requires common upper and lower bounds, whereas a lattice order requires least upper and greatest lower bounds. Every lattice order is directed, but the converse fails; maintaining this distinction is essential when comparing the subsequent existence results.

An important early development concerned the real field. Wilson constructed non-total lattice orders on subfields of $\mathbb{R}$ and proved that $\mathbb{R}$ admits $2^{\mathfrak c}$ distinct lattice orders, where $\mathfrak c=2^{\aleph_0}$ is the cardinality of the continuum~\cite{Wilson1976}. Thus the usual total order does not exhaust the lattice orders compatible with real-field arithmetic. This result established considerable flexibility on the real side, but it did not supply an order on $\mathbb{C}$, where the presence of $i$ prevents the use of total-order arguments. Schwartz subsequently investigated the structure of lattice-ordered fields, particularly their relation to totally ordered subfields and to algebraic extensions~\cite{Schwartz1986}. As recalled by Schwartz and Yang~\cite{SY2023}, this work also gives an obstruction in the algebraic-number setting: a number field admitting a compatible lattice order must be real.

The directed version of the existence problem developed further through constructions allowing negative squares. Yang proved the existence of directed rings and algebras with this feature~\cite{Yang2006}. Schwartz and Yang then established the existence of compatible directed partial orders on $\mathcal{C}$~\cite{SY2011}. Their later characterization states that a field admits a compatible directed partial order precisely when it has characteristic zero and is either formally real or has positive transcendence degree over $\mathbb{Q}$~\cite[Theorem~4]{SY2023}. In particular, transcendence permits directed orders on fields that cannot be totally ordered. This criterion settles the directed existence question, but it does not provide joins and meets.

Other work has studied partial orders on quadratic extensions and quaternion algebras while keeping a specified ordered coefficient domain fixed; see, for example, Ma~\cite{Ma2015}. More recently, Wang, Yuan, Zhang, and Zhu constructed explicit directed positive cones on $\mathbb{C}$ using integral closures of local rings~\cite{WYYZ2026}. Their nonzero positive elements are closed under inversion. They proved that none of the orders in that family is a lattice order, since $0$ and $i$ have no least upper bound. Their September 2026 preprint accordingly still records the general lattice-order problem as open and conjectures a negative answer. The failure of joins for those cones is a statement about that construction, rather than a general obstruction to all compatible partial orders on $\mathbb{C}$.

In the present paper, we give an affirmative answer to the Birkhoff--Pierce problem. We construct lattice orders on the usual complex field and prescribe specific lattice identities involving the original complex numbers $1$, $i$, and a chosen transcendental real parameter. More precisely, our principal result is the following.

\begin{theorem}
For every real number $t>1$ transcendental over $\mathbb{Q}$, there exists a partial order $\preceq$ on the usual field $\mathbb{C}$ that is compatible with addition and multiplication and makes $\mathbb{C}$ a lattice. Its restriction to $\mathbb{Q}$ is the usual order, and it satisfies $0\prec1$, $0\vee i=t$, and $1\wedge t=0$. Moreover, the construction preserves the order and the binary lattice operations on a specified countable algebraically closed subfield containing $\mathbb{Q}(t,i)$.
\end{theorem}

Here $x\prec y$ means $x\preceq y$ and $x\neq y$. Since the value of $0\vee i$ is determined by the order relation, distinct parameters $t$ produce distinct compatible lattice orders. There are $\mathfrak c$ real transcendental numbers greater than $1$, so $\mathbb{C}$ admits at least $\mathfrak c$ distinct such orders. This is a lower bound for distinct order relations on the fixed complex field; no classification up to field automorphism is asserted.

Our approach separates the construction into a countable algebraic stage and a passage to the full complex field. Fix the parameter $t$ and put $F=\mathbb{Q}(t)$. Let $K$ be the set of elements of $\mathbb{C}$ algebraic over $F$. Then $K$ is countable and algebraically closed. Countability makes it possible to enumerate all its elements and to satisfy their spanning requirements one at a time. At this stage the objective is stronger than the construction of a directed cone: we seek a $\mathbb{Q}$-basis $B$ of $K$ for which every product of basis elements is a finite nonnegative rational linear combination of basis elements.

Such a basis immediately gives a lattice order. Its nonnegative coordinate cone $P_B$ defines $x\leq_K y$ by $y-x\in P_B$, and the join and meet are obtained by taking coordinatewise maxima and minima. Multiplicative compatibility is equivalent to the required nonnegativity of the multiplication coefficients on basis products. An arbitrary basis offers no reason for those coefficients to be nonnegative. The main constructive issue is therefore to enlarge a basis while imposing finitely many positivity requirements and preserving all coordinate calculations already made.

We resolve this issue by a finite basis-extension lemma. Suppose an existing basis spans a subspace $V$, and a finite set of new targets lies in an affine hyperplane on which a rational linear functional has value $1$, while the functional vanishes on $V$. A suitable change of the new basis vectors, together with a common correction in the old nonnegative cone, places all targets in the enlarged nonnegative cone. The enlarged basis contains the original basis, and the new cone intersects $V$ in exactly the old cone. This finite-dimensional calculation is the mechanism that preserves earlier stages. Its affine normalization also explains the role of valuation theory in the argument.

On $F$, the degree of a nonzero rational function is $\deg p-\deg q$ for $\frac{p(t)}{q(t)}$ with nonzero $p,q\in\mathbb{Q}[x]$. The negative of this function is the valuation at infinity. Extending it to the algebraic field $K$ gives a rational-valued degree function $\delta$, with $\delta(t)=1$, $\delta(q)=0$ for $q\in\mathbb{Q}^\times$, $\delta(xy)=\delta(x)+\delta(y)$, and $\delta(x+y)\leq\max\{\delta(x),\delta(y)\}$. Standard valuation-extension results and the torsion property of value-group extensions in algebraic field extensions explain why the values can be normalized in $\mathcal{Q}$; see Engler and Prestel~\cite{EP2005}. We also give the requisite construction through a dominating valuation ring.

Write $\mathcal{O}=\{x\in K:\delta(x)\leq0\}$ and $m=\{x\in K:\delta(x)<0\}$, with $\delta(0)=-\infty$. The residue field $k=\mathcal{O}/m$ contains a canonical copy of $\mathbb{Q}$. For each nonnegative integer $n$, the subspace $K_n=\{x\in K:\delta(x)\leq n\}$ carries a linear map $\ell_n(x)=\operatorname{res}(x/t^n)$ into $k$. This map vanishes exactly on the elements of degree less than $n$ and satisfies $\ell_n(t^n)=1$. Furthermore, the maps respect multiplication across the filtration: $\ell_{r+s}(xy)=\ell_r(x)\ell_s(y)$ when $x\in K_r$ and $y\in K_s$. These properties allow products and spanning targets to be normalized by the same functional.

The inductive construction produces finite layers $B_n$, starting with $B_0=\{1\}$ and $B_1=\{t,t-i\}$, such that $\delta(b)=n$ and $\ell_n(b)=1$ for $b\in B_n$. At stage $n$, we impose positivity on $t^n$ and on the products $bc$ with $b\in B_r$, $c\in B_s$, $r,s\geq1$, and $r+s=n$. Both factors then belong to earlier stages, so these product requirements introduce no circular dependence. To ensure that the final basis spans $K$, choose an enumeration $K=\{x_1,x_2,\ldots\}$ and strictly increasing positive integers $N_j>\delta(x_j)$. At stage $N_j$, include the additional target $t^{N_j}-x_j$. Once both this target and $t^{N_j}$ belong to the constructed span, so does $x_j$.

Every target at stage $n$ has $\ell_n$-value $1$, while every element of the old span has $\ell_n$-value zero. Although $k$ need not equal $\mathbb{Q}$, the restriction of $\ell_n$ to the space generated by the old span and these targets takes values in the canonical copy of $\mathbb{Q}$. The finite basis-extension lemma therefore applies over precisely the required scalar field. At the limit, every element of $K$ has entered the span and every product of basis elements has entered the nonnegative cone. The union $B=\bigcup_{n\geq0}B_n$ is consequently a positive basis with the desired multiplication property. The initial choice of $B_1$ yields $i=t-(t-i)$, and the coordinate lattice formulas give $0\vee_K i=t$ and $1\wedge_K t=0$.

It remains to obtain an order on all of $\mathbb{C}$. Rather than carry out the same finite scheduling directly through an uncountable construction, we pass to the ultrapower $E=K^\mathds{N}/\mathcal{U}$ for a nonprincipal ultrafilter $\mathcal{U}$ on $\mathbb{N}$. This step preserves the field operations, the order relation, and the binary lattice operations simultaneously. The usual preservation principle for ultraproducts provides the general model-theoretic background~\cite{Marker2002}; here the required properties are verified directly on representatives. In particular, $E$ is algebraically closed of characteristic zero, and the diagonal embedding $j:K\longrightarrow E$ preserves and reflects order and preserves joins and meets.

The cardinality calculation is equally important. Countability of $K$ gives $|E|\leq\mathfrak c$, and a coding of binary sequences by their finite initial segments gives the reverse inequality. Thus $|E|=\mathfrak c$. Since $K$ is countable, both $E$ and $\mathbb{C}$ have transcendence degree $\mathfrak c$ over their respective copies of $K$. Matching transcendence bases gives an isomorphism of the corresponding rational function fields. An embedding-extension argument for algebraic extensions then produces a field embedding $\Phi:E\longrightarrow\mathbb{C}$ satisfying $\Phi(j(a))=a$ for every $a\in K$. Its image is algebraically closed, and $\mathbb{C}$ is algebraic over that image, so $\Phi$ is surjective.

Finally, transport the lattice order of $E$ through $\Phi$. Because $\Phi$ is a field isomorphism, the transported order is compatible with the original addition and multiplication on $\mathcal{C}$. Because $\Phi$ fixes the diagonal copy of $K$, the prescribed identities for the actual numbers $t$ and $i$ survive the transport. Fixing $K$ is therefore essential: an unspecified abstract isomorphism would not by itself ensure those identities for the chosen complex numbers. The proof combines valuation rings, finite extensions of bases in infinite-dimensional rational vector spaces, nonprincipal ultrafilters, transcendence bases, and extension of field embeddings.

The scope of the resulting order is reflected in its prescribed identities. Since $t>1$ in the ordinary real order but $1\wedge t=0$ in our lattice, the order does not extend the usual order on all of $\mathbb{R}$. It does extend the usual rational order. Also, $t$ is positive whereas neither $t^{-1}$ nor $-t^{-1}$ is nonnegative. On the countable field this follows from the fact that every nonzero element of $P_B$ has nonnegative integral degree, while $\delta(t^{-1})=-1$, and it remains true after transport. Thus closure under inversion is absent, distinguishing the construction from the inversion-closed directed cones discussed above. The argument is algebraic and uses choice through maximality arguments and the ultrafilter construction.

The paper develops these ingredients in the order required by the proof. We first establish the finite positive basis-extension principle and the valuation and residue calculations. We then construct the countable positive basis and describe its coordinate lattice order, including the initial layers and the degree restrictions on positive elements. The ultrapower construction and the field isomorphism fixing $K$ provide the passage to $\mathbb{C}$. The final results give the main theorem and explicit formulas for comparisons, joins, and meets in terms of representatives and rational basis coordinates.

\section{Preliminaries}

We collect the terminology and algebraic tools used in the construction. All rings are commutative with identity, all subrings contain the ambient identity, and all field embeddings preserve it. We write $\mathbb{N}_0=\{0,1,2,\ldots\}$ and $\mathfrak{c}=2^{\aleph_0}$.

\subsection{Positive cones and compatible lattice orders}

A \emph{positive cone} in a field $L$ is a subset $P$ such that $0\in P$, $P+P\subseteq P$, $PP\subseteq P$, and $P\cap(-P)=\{0\}$. It defines a partial order by $x\leq y$ if and only if $y-x\in P$. This order is translation invariant, and products of nonnegative elements are nonnegative. We call it a \emph{compatible lattice order} if every pair has a least upper bound $x\vee y$ and a greatest lower bound $x\wedge y$. A directed order requires only common upper and lower bounds. Lattice orders are directed; they need not be total, and positivity of inverses is not assumed. These conventions concern lattice-ordered rings whose underlying rings are fields~\cite{Fuchs1963}.

\subsection{Rational bases and finite positive extensions}

A $\mathbb{Q}$-basis $B$ of a vector space $V$ is a Hamel basis: every $x\in V$ has a unique expansion $x=\sum_{b\in B}x_b b$, with $x_b\in\mathbb{Q}$ and finite support. Its coordinate cone is $P_B=\{\sum_{b\in B}q_b b:q_b\in\mathbb{Q}_{\geq0},\ \text{finite support}\}$. Then $x\leq_B y$ is equivalent to $x_b\leq y_b$ for every $b\in B$. The lattice operations are $x\vee_B y=\sum_b\max\{x_b,y_b\}b$ and $x\wedge_B y=\sum_b\min\{x_b,y_b\}b$. Both sums have support contained in the union of the supports of $x$ and $y$.

If $1\in B$, the order on $\mathbb{Q}$ is the usual one. When $V$ is a field, the condition $bc\in P_B$ for all $b,c\in B$ guarantees $P_BP_B\subseteq P_B$ by distributivity. The finite extension lemma proved in our construction asserts the following. If $B$ is a basis of $V$, $A$ is a nonempty finite subset of a larger rational vector space, $W=V+\operatorname{span}_{\mathbb{Q}}A$, and a rational linear map $\lambda:W\to\mathbb{Q}$ satisfies $\lambda(V)=0$ and $\lambda(a)=1$ for $a\in A$, then a finite set $B'$ can be chosen so that $B\cup B'$ is a basis of $W$, $\lambda(b')=1$ for $b'\in B'$, $A\subseteq P_{B\cup B'}$, and $P_{B\cup B'}\cap V=P_B$. Thus finitely many positivity requirements can be imposed without altering the old coordinates.

\subsection{Local rings, valuation rings, and residue maps}

A \emph{local ring} $(R,\mathfrak m_R)$ has a unique maximal ideal; equivalently, its nonunits form the ideal $\mathfrak m_R$. A module $M$ is finitely generated if $M=Ru_1+\cdots+Ru_r$ for finitely many elements $u_j$. Nakayama's lemma states that $\mathfrak m_RM=M$ then implies $M=0$. A local subring $(S,\mathfrak n)$ of a field dominates $(R,\mathfrak m_R)$ if $R\subseteq S$ and $\mathfrak n\cap R=\mathfrak m_R$. Every local subring of a field is dominated by a valuation ring in that field. A valuation ring $\mathcal O$ of $L$ is characterized by the property that $x\in\mathcal O$ or $x^{-1}\in\mathcal O$ for every $x\in L^\times$.

A rational-valued valuation $\nu:L^\times\to\mathbb{Q}$ is additive on products and satisfies $\nu(x+y)\geq\min\{\nu(x),\nu(y)\}$ when $x+y\neq0$; set $\nu(0)=+\infty$. For the present application, fix a real transcendental $t>1$, put $F=\mathbb{Q}(t)$, and let $K\subseteq\mathbb{C}$ consist of the elements algebraic over $F$. The field $K$ is countable and algebraically closed. The valuation at infinity on $F$, given by $\nu_0(p(t)/q(t))=\deg q-\deg p$ for nonzero $p,q\in\mathbb{Q}[X]$, extends to a $\mathbb{Q}$-valued valuation $\nu$ on $K$~\cite{EP2005}. Set $\delta=-\nu$ and $\delta(0)=-\infty$. Then $\delta(xy)=\delta(x)+\delta(y)$ for nonzero $x,y$, and $\delta(x+y)\leq\max\{\delta(x),\delta(y)\}$. Moreover, $\delta(t)=1$ and $\delta(q)=0$ for $q\in\mathbb{Q}^\times$.

The associated valuation ring and maximal ideal are $\mathcal O=\{x:\delta(x)\leq0\}$ and $\mathfrak m=\{x:\delta(x)<0\}$. Write $k=\mathcal O/\mathfrak m$ and $\operatorname{res}:\mathcal O\to k$ for the residue map. Its restriction to $\mathbb{Q}$ is injective. For $n\in\mathbb{N}_0$, put $D_n=\{x\in K:\delta(x)\leq n\}$ and $\ell_n(x)=\operatorname{res}(x/t^n)$. This is a well-defined $\mathbb{Q}$-linear map on $D_n$, with $\ell_n(x)=0$ exactly when $\delta(x)<n$, and $\ell_n(t^n)=1$. It also satisfies $\ell_{r+s}(xy)=\ell_r(x)\ell_s(y)$ for $x\in D_r$, $y\in D_s$. These identities supply the normalized functionals for the finite basis extensions.

\subsection{Nonprincipal ultrafilters and ultrapowers}

An \emph{ultrafilter} $\mathcal U$ on $\mathbb{N}$ is a proper filter such that, for every $A\subseteq\mathbb{N}$, exactly one of $A$ and $\mathbb{N}\setminus A$ belongs to $\mathcal U$. Here a proper filter contains $\mathbb{N}$, excludes $\varnothing$, and is closed under finite intersections and passage to supersets. It is nonprincipal if it contains no finite set. Such ultrafilters exist by extending the cofinite filter using Zorn's lemma. Define $(x_n)\sim(y_n)$ when $\{n:x_n=y_n\}\in\mathcal U$, and write $E=K^{\mathbb{N}}/\mathcal U$ for the resulting ultrapower.

Addition and multiplication descend coordinatewise, making $E$ a field. If $K$ is algebraically closed, so is $E$: roots of coordinate polynomials give a root in the quotient. A compatible lattice order on $K$ induces one on $E$ by $[(x_n)]\leq_E[(y_n)]$ if $\{n:x_n\leq_Ky_n\}\in\mathcal U$. Joins and meets are computed coordinatewise. The diagonal map $j(a)=[(a,a,\ldots)]$ is a field embedding preserving and reflecting order and preserving the lattice operations. These are instances of the ultraproduct preservation principle~\cite{Marker2002}. For countably infinite $K$ and nonprincipal $\mathcal U$ on $\mathbb{N}$, one has $|E|=\mathfrak c$.

\subsection{Transcendence bases and extension of field embeddings}

A set $T\subseteq L$ is algebraically independent over a subfield $k_0$ if no nonzero polynomial over $k_0$ in finitely many distinct elements of $T$ vanishes. A maximal such set is a \emph{transcendence basis}; then $L/k_0(T)$ is algebraic, and $|T|=\operatorname{trdeg}_{k_0}L$. If $k_0$ is countable and $L$ is uncountable, then $|L|=\operatorname{trdeg}_{k_0}L$. Indeed, rational functions involve finitely many variables, and algebraic extensions add at most finitely many roots for each polynomial.

We also use the embedding-extension theorem: if $L/D$ is algebraic, $\Omega$ is algebraically closed, and $\sigma:D\hookrightarrow\Omega$ is a field embedding, then $\sigma$ extends to an embedding $L\hookrightarrow\Omega$. It follows that algebraically closed extensions of a fixed field with equally sized transcendence bases are isomorphic over that field. Here, matching transcendence bases over $K$ and extending the induced rational-function-field isomorphism gives $\Phi:E\to\mathbb{C}$ with $\Phi(j(a))=a$ for $a\in K$. Surjectivity follows because the image is algebraically closed and $\mathbb{C}$ is algebraic over it. Transporting the order through $\Phi$ preserves both field compatibility and the prescribed lattice identities on $K$.

\section{Conventions and finite positive basis extensions}

All rings are commutative and have an identity. The symbols $\leq$ and $\geq$ on rational coefficients denote the usual order of $\mathbb{Q}$.

\begin{lemma}{\rm Let $P\subseteq R$ satisfy

(1) $0\in P$;

(2) $P+P\subseteq P$;

(3) $P\cap(-P)=\{0\}$;

(4) $PP\subseteq P$.

Then $\preceq$ is a partial order on $P$, where $\preceq$ is given by $x\preceq y$ if and only if $y-x\in P$.}
\end{lemma}

\begin{proof} It is straightforward by checking all requirements.
\end{proof}

Given a $\mathbb{Q}$-linearly independent set $B$, write
$$P_B=\operatorname{Cone}_{\mathbb{Q}}(B)
=\{\sum_{b\in B}q_b b\mid q_b\in\mathbb{Q}_{\geq0},\ \operatorname{supp}(q)\text{ finite}\},
$$ where $\operatorname{supp}(q)=\{q_{b}\in\mathbb{Q}\mid q_{b}\neq 0\}$ and $\mathbb{Q}_{\geq0}$ is the set of all nonnegative rational numbers.

\begin{lemma} {\rm (\cite{RTY1}) Let $V$ be a subspace of a vector space $W$. If $B$ is a basis of $V$ and the set $\{b_{1}+V,\cdot\cdot\cdot, b_{n}+V\}$ is a basis of $W/V$, then $B\cup\{b_{1},\cdot\cdot\cdot, b_{n}\}$ is a basis of $W$.}

\end{lemma}

\begin{lemma}
{\rm Let $B$ be a basis of a $\mathbb{Q}$-vector space $V$ and $F$ a nonempty finite subset of a $\mathbb{Q}$-vector space $E$ containing $V$.  Put
$$
W=V+\operatorname{span}_{\mathbb{Q}} F.
$$
If $\lambda:W\to\mathbb{Q}$ is a $\mathbb{Q}$-linear map satisfies $\lambda(V)=\{0\},\lambda(F)=\{1\}$. Then there is a finite set $B'=\{b_1,\ldots,b_n\}$ such that

(1) $B\cup B'$ is a basis of $W$;

(2) $\lambda(b_j)=1$ for every $1\leq j\leq n$;

(3) $F\subseteq P_{B\cup B'}$ and $P_{B\cup B'}\cap V=P_B$.}
\end{lemma}

\begin{proof} As $W=V+\operatorname{span}_{\mathbb{Q}} F$, we can choose $e_{1}+V,\ldots,e_{n}+V$ forms a basis of the quotient space $W/V$ for some $\{e_{1},\ldots,e_{n}\}\subseteq F$. Since $\lambda(V)=\{0\}$ and $\lambda(F)=\{1\}$, there exists $f\in F\setminus V$. Then we have $f+V\neq 0_{W\setminus V}$ and hence $1\leq diam(W/V)=n$.
For every $a\in F$, $a+V$ can be written uniquely as
$$a+V=\sum_{j=1}^{n}k_{aj}(e_j+V),k_{aj}\in\mathbb{Q}.$$ It yields that $a-\sum\limits_{j=1}\limits^{n}k_{aj}e_j\in V$. Define $v_{a}=a-\sum\limits_{j=1}\limits^{n}k_{aj}e_j$. Now, $a$ can be uniquely expressed as $a=v_{a}+\sum\limits_{j=1}\limits^{n}k_{aj}e_j$. As a consequence, $$\lambda(\upsilon_{a})+\sum\limits_{j=1}\limits^{n}k_{aj}\lambda(e_j)=\sum\limits_{j=1}\limits^{n}k_{aj}=\lambda(a)=1.$$
Fix $M\in\mathbb{N}$ such that for every $1\leq j\leq n$ and for every $a\in F$, $0\leq k_{aj}+M$. Since $B$ is a basis of $V$, there exists a finite subset $S_{a}\subseteq B$ such that $v_a=\sum\limits_{b\in B}v_{a,b}b$, where each $v_{a,b}$ is a rational number and $v_{a,b}=0$ if $b\in B\setminus S_{a}$. For each $b\in B$, set
$$h_b=\max\bigl(\{0\}\cup\{-v_{a,b}\mid a\in F\}\bigr)\text{ and }h=\sum_{b\in B}h_b b.$$ One can check that for each $a\in F$,
$$h\in P_B\text{ and } v_a+h\in P_B.$$
For each $j$, take
$$
G=1+nM,b_j=Ge_j-M\sum_{r=1}^{n}e_r-h.
$$

Set $B'=\{b_{1},\cdot\cdot\cdot,b_{n}\}$.

(1) $B\cup B'$ is a basis of $W$.

By Lemma 3.2, it suffices to prove that $b_{1}+V,\cdot\cdot\cdot,b_{n}+V$ forms a basis of $W/V$.

Since $h\in V$, we have $$b_{j}+V=Ge_j-M\sum_{r=1}^{n}e_r-h+V=-\sum\limits_{r=1}\limits^{j-1}M(e_{r}+V)+(G-M)(e_{j}+V)-\sum\limits_{r=j+1}\limits^{n}M(e_{r}+V).$$
Consequently,
\[
\begin{pmatrix}
b_1 + V \\
b_2 + V \\
\vdots \\
b_n + V
\end{pmatrix}
=
T
\begin{pmatrix}
e_1 + V \\
e_2 + V \\
\vdots \\
e_n + V
\end{pmatrix},
\]
where the matrix $T$ is given by
\[
T = G I_n - M J_n =
\begin{pmatrix}
G-M    & -M     & \cdots & -M \\
-M     & G-M    & \cdots & -M \\
\vdots & \vdots & \ddots & \vdots \\
-M     & -M     & \cdots & G-M
\end{pmatrix}.
\]

Suppose $(GI_n-MJ_n)u=0$, $u\in\mathbb{Q}^{n}$. Then for every $1\leq j\leq n$,
$$Gu_j=M\sum_{k=1}^{n}u_k.$$
This yields $(G-nM)\sum\limits_{k=1}\limits^{n}u_k=0$. We obtain $\sum\limits_{k=1}\limits^{n}u_k=0$.
It follows form $G\ne0$ that each $u_{j}=0$. This means that $T$ is a invertible matrix. Hence $b_{1}+V,\cdot\cdot\cdot,b_{n}+V$ is a basis of the quotient space $W/V$.

(2) $\lambda(b_{j})=1$ for each $1\leq j\leq n$.

Indeed, $$\lambda(b_j)=\lambda(Ge_j-M\sum_{r=1}^{n}e_r-h)=G\lambda(e_j)-M\sum_{r=1}^{n}\lambda(e_r)-\lambda(h)=G-nM=1.$$

(3) $F\subseteq P_{B\cup B'}$ and $P_{B\cup B'}\cap V=P_B$.

Choose any $a\in F$, put $\alpha_{aj}=\frac{k_{aj}+M}{G}$.
Then we have $\alpha_{aj}\geq0$ and $\sum\limits_{j=1}\limits^{n}\alpha_{aj}=\frac{\sum\limits_{j=1}\limits^{n}k_{aj}+nM}{G}=1$
Furthermore,
$$
\begin{aligned}
\sum_{j=1}^{n}\alpha_{aj}b_j
&=\sum_{j=1}^{n}\alpha_{aj}(Ge_j-M\sum_{r=1}^{n}e_r-h)\\
&=\sum_{j=1}^{n}(k_{aj}+M)e_j
-M\sum_{k=1}^{n}e_k-h\\
&=\sum_{j=1}^{n}k_{aj}e_j-h.
\end{aligned}
$$
Now, we get $a=(v_a+h)+\sum\limits_{j=1}\limits^{n}\alpha_{aj}b_j\in P_{B\cup B'}$.

Clearly, $P_{B}\subseteq P_{B\cup B'}\cap V$. Take any $v\in P_{B\cup B'}\cap V$. Then $v$ can be uniquely express as $\sum\limits_{b\in B}q_{v,b}b+\sum\limits_{j=1}\limits^{n}q_{v,j}b_{j}$, where $\operatorname{supp}(q)=\{q_{v,b}\mid q_{v,b}\neq 0\}$ is finite and all $q_{v,b},q_{v,j}$ are positive rational numbers. Applying the mapping $\lambda$, we have
$$0=\lambda(v)=\sum\limits_{b\in B}q_{v,b}\lambda(b)+\sum\limits_{j=1}\limits^{n}q_{v,j}\lambda(b_{j})=\sum\limits_{j=1}\limits^{n}q_{v,j}.$$
As a result, $q_{v,j}=0$ for each $1\leq j\leq n$. Equivalently, $v\in P_{B}$. Thus, $P_{B\cup B'}\cap V=P_{B}$.\end{proof}

\section{A countable algebraically closed field and a rational degree function}

Fix a real number $t>1$, $t$ is transcendental over $\mathbb{Q}$.
Set a filed $F=\mathbb{Q}(t)$ and
$$K=\{z\in\mathbb{C}:z\text{ is algebraic over }F\}.$$

\begin{lemma} {\rm Let $F$ be an extension of $E$ and $\alpha\in F$. Then $\alpha$ is algebraic over $F$ if and only if $[E(\alpha):E]<\infty$.}
\end{lemma}

\begin{lemma} {\rm  Let $E$ be an extension of $F$ and $K$ an extension of $E$. If $[E:F]$ and $[K:E]$ are finite, then $[K:F]$ is finite and further $[K:F]=[E:F][K:E]$.}
\end{lemma}

\begin{lemma}
{\rm $K$ is a countable algebraically closed subfield of $\mathbb{C}$ and contains $i=\sqrt{-1}$.}
\end{lemma}

\begin{proof} Clearly, $i^{2}+1=0$ and hence $i\in K$. Let $a,b\in K$. Equivalently, $a,b$ are algebraic over $F$. By Lemma 4.2, we have $$[F(a,b):F]=[F(a,b):F(a)][F(a):F]$$ is finite. So each element of the filed $F(a,b)$ is algebraic over $F$. Hence $a\pm b$, $ab$ and $a^{-1}$ for $a\ne0$ contained in $F(a,b)$ belong to $K$. Thus $K$ is a subfield of $\mathbb{C}$.

The field $$F=\mathbb{Q}(t)=\{\frac{f(t)}{g(t)}\mid f,g\in\mathbb{Q}[x], g\neq0\}$$ is countable.
This leads to $F[x]$ is also countable. In the filed $\mathbb{C}$, every polynomial with degree at least 1 has finitely many roots. Thus $|K|\leq\aleph_0$. $F\subseteq K$ ensures that $|K|=\aleph_0$.

Suppose $f\in K[x]$ is nonconstant. For convenience, we let $$f(x)=\sum\limits_{k=0}\limits^{n}a_{k}x^{k}\in K[x]$$ and set $L=F(a_{0},a_{1},\cdot\cdot\cdot,a_{n})$. By Lemma 4.2, $$[L:F]=[F(a_{0}),F]\prod\limits_{k=1}\limits^{n-1}[F(a_{0},a_{1},\cdot\cdot\cdot,a_{k+1}):F(a_{0},a_{1},\cdot\cdot\cdot,a_{k})]$$ is finite. As $f\in\mathbb{C}[x]$ and $\mathbb{C}$ is algebraically closed, we can fix a $z\in\mathbb{C}$ with $f(z)=0$. It follows form $f\in L[x]$ that $z$ is algebraic over $L$. By Lemma 4.1, $[L(z):L]<\infty$. As a result, $[L(z):F]=[L(z):L][L:F]$ is finite. So $z$ is algebraic over $F$. In other words, $z\in K$. Thus $K$ is algebraically closed.
\end{proof}

\begin{lemma}{\rm (\cite{JRRG13})Let $(R,\mathfrak{m})$ be a local ring and $M$ a finitely generated $R$-module with $\mathfrak{m} M=M$. Then $M=0$.}
\end{lemma}

\begin{lemma}
{\rm Let $(R,\mathfrak{m}_R)$ be a local subring of a field $L$. There exists a local subring $(\mathcal{O},\mathfrak{m})$ of $L$ such that

(1) $R\subseteq\mathcal{O}$ and $\mathfrak{m}\cap R=\mathfrak{m}_R$;

(2) $x\in L\setminus\{0\}$ implies $x\in\mathcal{O}$ or $x^{-1}\in\mathcal{O}$.}
\end{lemma}

\begin{proof}
Set $$\mathscr{A}=\{(S,\mathfrak{m}_S)\mid(S,\mathfrak{m}_S)\text{ is a local subring of } L\text{ and }R\subseteq S,\mathfrak{m}_S\cap R=\mathfrak{m}_R\}$$
and the order $\mathscr{A}$ is defined by:
$$(S_1,\mathfrak{m}_{S_{1}})\preceq (S_2,\mathfrak{m}_{S_{2}})\Longleftrightarrow S_1\subseteq S_2\text{ and }\mathfrak{m}_{S_2}\cap S_1=\mathfrak{m}_{S_1}.$$
Let $\{((S_i,\mathfrak{m}_{S_{i}}))\mid i\in I\}$ be a nonempty chain of $\mathscr{A}$. Take
$$
S=\bigcup_j S_j,\text{ and } \mathfrak n=\bigcup_j\mathfrak{m}_{S_j}.
$$
Then $S$ is a subring of $L$ and $\mathfrak n$ is an ideal of $S$. Furthermore, $1\not\in\mathfrak n$. In other words, $\mathfrak n$ is a proper ideal. Choose any $s\in S\setminus\mathfrak n$. We can fix $i$ with $s\in S_i$. So $s\notin\mathfrak m_i$ and hence $s$ is a unit of $S_i$. As a consequence, $s^{-1}\in S_i\subseteq S$.
Thus every element of $S\setminus\mathfrak n$ is a unit of $S$.
This proves that $\mathfrak n$ contains no units.
Hence, $(S,\mathfrak n)$ is a local subring.

We claim that for each $i$, $\mathfrak n\cap S_i=\mathfrak m_{S_{i}}$. Suppose $a\in\mathfrak n\cap S_i$. Select $j$ such that $a\in\mathfrak m_{S_{j}}$. If $(S_j,\mathfrak m_{S_{j}})\preceq(S_i,\mathfrak m_{S_{i}})$, then
$a\in\mathfrak m_{S_{j}}\subseteq\mathfrak m_{S_{i}}$.
On the contrary,
$$
a\in\mathfrak m_{S_{j}}\cap S_i=\mathfrak m_{S_{i}}.
$$
All in all, $a\in\mathfrak m_{S_{i}}$. Consequently, $\mathfrak n\cap S_i=\mathfrak m_{S_{i}}$ holds. Choose a $k\in I$,
$$
\mathfrak n\cap R
=
(\mathfrak n\cap S_{k})\cap R
=m_{S_{k}}\cap R
=
\mathfrak m_R.
$$
Therefore, $(S,\mathfrak n)\in\mathscr A$. Using Zorn's lemma, $\mathscr A$ has a maximal element $(\mathcal O,\mathfrak m)$.
Then, we have
$$
R\subseteq\mathcal O\subseteq L
\text{ and }
\mathfrak m\cap R=\mathfrak m_R.
$$

Choose any $x\in (L\setminus\{0\})\setminus\mathcal O$.
We first prove that $$\mathfrak m\mathcal O[x]=\{\sum\limits_{k=0}\limits^{n}a_{k}x^{k}\mid 0\leq n,\text{ each }a_{k}\in\mathfrak m\}=\mathcal O[x].$$
One can check that $\mathfrak m\mathcal O[x]$ is an ideal of $\mathcal O[x]$. Select a maximal ideal $\mathfrak q$ of
$\mathcal O[x]$ such that $\mathfrak m\mathcal O[x]\subseteq\mathfrak q$.
This means that the ideal $\mathfrak q\cap\mathcal O$ is proper and
contains $\mathfrak m$. The maximality of $\mathfrak m$ gives
$\mathfrak q\cap\mathcal O=\mathfrak m$.

Set
$$
B
=
\left\{
fg^{-1}\in L\mid
f,g\in\mathcal O[x],\ g\notin\mathfrak q
\right\}.
$$
Since $\mathcal O[x]$ is a subring of the field $L$ and $\mathfrak q$ is a maximal ideal, $B$ is
a subring of $L$. Furthermore, $\mathfrak qB=\{fg^{-1}\in L\mid f\in \mathfrak q,g\in O[x]\setminus\mathfrak q\}$ is a ideal of $B$ and contains all non unit of $B$. So $(B,\mathfrak qB)$ is a local ring. Clearly, $\mathfrak q\subseteq(\mathfrak qB)\cap O[x]$. If $a\in (\mathfrak qB)\cap O[x]$, then $a=fg^{-1}$ for some $f\in\mathfrak q$ and $g\not\in\mathfrak q$. Then $ga=f\in\mathfrak q$. As $\mathfrak q$ is a maximal ideal and $g\not\in\mathfrak q$, $a\in\mathfrak q$. Thus $(\mathfrak qB)\cap O[x]=\mathfrak q$. As a consequence,
$$
(\mathfrak qB)\cap\mathcal O
=
((\mathfrak qB)\cap O[x])\cap\mathcal O
=
\mathfrak q\cap\mathcal O
=\mathfrak m.
$$

 Moreover,
$$
(\mathfrak qB)\cap R
=((\mathfrak qB)\cap\mathcal O)\cap R
=\mathfrak m\cap R
=
\mathfrak m_R,
$$
Consequently, $(B,\mathfrak qB)\in\mathscr A$ and $(\mathcal O,\mathfrak m)\preceq(B,\mathfrak qB)$. Note that $x\in B$ but $x\notin\mathcal O$. This contradicts with the fact that $(\mathcal O,\mathfrak m)$ is maximal in $\mathscr A$.
As a result, $\mathfrak m\mathcal O[x]=\mathcal O[x]$ is valid.

So we obtain $1\in\mathfrak m\mathcal O[x]$. In other words,
$$\exists\ a_0,\ldots,a_n\in\mathfrak m \text{ such that }a_0+a_1x+\cdots+a_nx^n=1.$$
It follows form $1\not\in\mathfrak m$ that $1\leq n$.

Let $y=x^{-1}$. Then we have $$(1-a_0)y^n=a_1y^{n-1}+\cdots+a_n.$$ As $a_0\in\mathfrak m$ and $(\mathcal O,\mathfrak m)$ is local subring,
$1-a_0$ is a unit of $\mathcal O$.
Set $c_i=(1-a_0)^{-1}a_i\in\mathcal O$ ($1\leq i\leq n$). So we get
$$y^n=c_1y^{n-1}+\cdots+c_n.$$

Let $T=\mathcal O[y]$. By induction, each $y^{m}$ ($n\leq m$) is an
$\mathcal O$-linear combination of $y,y^{2}\cdots,y^{n-1}$. So we conclude that
$T$ is a finitely generated $\mathcal O$-module.
Clearly, $T\neq0$ because $1\in T\subseteq L$.

By Lemma 4.1, $\mathfrak mT\neq T$. Fix a maximal ideal $\mathfrak q'$ of $T$ such that
$\mathfrak mT\subseteq\mathfrak q'$. It follows form $1\in T$ that $\mathfrak q'\cap\mathcal O$ is a ideal of $\mathcal O$ and contains $\mathfrak m$. The maximality of $\mathfrak m$ gives
$$
\mathfrak q'\cap\mathcal O=\mathfrak m.
$$
Consider the ring $$B'
=
\left\{
fg^{-1}\in L\mid
f,g\in\mathcal O[y],\ g\notin\mathfrak q'
\right\}.
$$

Using the method for checking the property of the local subring $(B,\mathfrak qB)$ above, we have $(B',\mathfrak q'B')$ belongs to $\mathscr{A}$ and further $(\mathcal O,\mathfrak m)\preceq(B,\mathfrak q')$. By the maximality of $(\mathcal O,\mathfrak m)$,
$B'=\mathcal O$. So we conclude that
$$
x^{-1}=y\in\mathcal O
$$ as $y\in O[y]\subseteq B'$.
Thus the requirement (2) is proved. \end{proof}

\begin{lemma} {\rm Every totally ordered abelian group is torsion-free.}

\end{lemma}

\begin{proposition}
{\rm There is a mapping $\delta:K\to\mathbb{Q}\cup\{-\infty\}$ satisfies the following statements:

(1) $\delta(0)=-\infty,\delta(t)=1$ and for each $q\in\mathbb{Q}\setminus\{0\}$, $\delta(q)=0$;

(2) For each $x,y\in K\setminus\{0\}$, $\delta(xy)=\delta(x)+\delta(y)$;

(3) For each $x,y\in K$, $\delta(x+y)\leq\max\{\delta(x),\delta(y)\}$;

(4) $p,q\in\mathbb{Q}[x]$, $\delta(\frac{p(t)}{q(t)})=\deg p-\deg q$.}
\end{proposition}

\begin{proof}
$\forall p,q\in\mathbb{Q}[x]$ with $\frac{p(t)}{q(t)}\neq0$, define
$$
\nu_0(\frac{p(t)}{q(t)})=\deg q-\deg p.
$$

The mapping $\delta$ is well-defined because $t$ is transcendental over $\mathbb{Q}$.
Let $f=\frac{e(t)}{h(t)}\in F\setminus\{0\}$ and $g=\frac{l(t)}{s(t)}\in F\setminus\{0\}$. Take $m=\min\{\nu_0(f),\nu_0(g)\}$. Then
$$
\begin{aligned}
\nu_0(fg)&=\deg(hs)-\deg(el)=\deg(h)-\deg(e)+\deg(s)-\deg(l)=\nu_0(f)+\nu_0(g),\\
\nu_0(f+g)&=\deg(hs)-\deg(es+lh)\geq\deg(h)+\deg(s)-\max\{\deg(es),\deg(lh)\} =\min\{\nu_0(f),\nu_0(g)\}.
\end{aligned}
$$
Set
$$
R=\{0\}\cup\{f\in F\setminus\{0\}\mid\nu_0(f)\geq0\}
\text{ and }
\mathfrak{m}_R=\{0\}\cup\{f\in F\setminus\{0\}\mid\nu_0(f)>0\}.
$$
Then $R$ is a subring of $K$ and $\mathfrak{m}_R$ is a ideal of $R$. Note that a element $f\in R$ is a unit if and only if $\nu_0(f)=0$. Thus $(R,\mathfrak{m}_R)$ is a local ring.

By Lemma 4.5, there is a local subring $(\mathcal{O},\mathfrak{m})$ such that

(1) $R\subseteq\mathcal{O}$ and $\mathfrak{m}\cap R=\mathfrak{m}_R$;

(2) $k\in K\setminus\{0\}$ implies $k\in\mathcal{O}$ or $k^{-1}\in\mathcal{O}$.

Choose any $f\in F\setminus R$. Then $f^{-1}\in\mathfrak{m}_R\subseteq\mathfrak{m}$. As $\mathfrak{m}$ is a ideal and $1=ff^{-1}\not\in\mathfrak{m}$, we obtain $f\notin\mathcal{O}$. This yields that
$$
\mathcal{O}\cap F=R.
$$

Consider the quotient group
$$
\Gamma=(K^\times/\mathcal{O}^\times,\bm{+}),
$$
and for each $x\in K^\times$, let $\widetilde\nu(x)$ be the class of $x\mathcal{O}^\times$, where $K^\times$ and $\mathcal{O}^\times$ denote the set of all units of $K$ and $\mathcal{O}$, respectively.  Define
$$
\widetilde\nu(x)\geq\widetilde\nu(y)
\quad\Longleftrightarrow\quad \frac{x}{y}\in\mathcal{O}.
$$
This is well-defined and $\leq$ forms a partial order. By (2), $\leq$ is a totally order.   Moreover, this order has the following properties.

(i) If $\widetilde\nu(x)\leq\widetilde\nu(y)$, then $\widetilde\nu(x)\bm{+}\widetilde\nu(z)\leq\widetilde\nu(y)\bm{+}\widetilde\nu(z)$ for any $z\in K^\times$.

By assumption, $\frac{x}{y}\in\mathcal{O}$. By the definition of the operation $\bm{+}$, $\widetilde\nu(x)\bm{+}\widetilde\nu(z)=\widetilde\nu(xz)$ and $\widetilde\nu(y)\bm{+}\widetilde\nu(z)=\widetilde\nu(yz)$. It follows form $\frac{x}{y}=\frac{xz}{yz}\in\mathcal{O}$ that $\widetilde\nu(x)\bm{+}\widetilde\nu(z)\leq\widetilde\nu(y)\bm{+}\widetilde\nu(z)$.

(ii) $$
\begin{aligned}
\widetilde\nu(x)=\widetilde\nu(1)\text{ if and only if } z\in\mathcal{O}^{+}\\
\widetilde\nu(x)\geq \widetilde\nu(1)\text{ if and only if } z\in\mathcal{O}\\
\widetilde\nu(x)>\widetilde\nu(1)\text{ if and only if } z\in\mathfrak{m}.\\
\end{aligned}
$$
It suffices to check the third equivalence relation. The condition $\widetilde\nu(x)>\widetilde\nu(1)$ means $z\not\in\mathcal{O}^{+}$ and $z\in\mathcal{O}$. In other words, $z\in\mathfrak{m}$. The converse is similar.

(iii) $\widetilde\nu(x+y)\geq\min\{\widetilde\nu(x),\widetilde\nu(y)\}$.

Without loss of generality, suppose $\widetilde\nu(x)\geq\widetilde\nu(y)$. Equivalently, $x/y\in\mathcal{O}$. Then we have $\frac{x+y}{y}=1+\frac{x}{y}\in\mathcal{O}$. This inequality is proved.

Set $$\gamma=\widetilde\nu(t^{-1}).$$

As $\frac{1}{t}\in\mathfrak{m}_R\subseteq\mathfrak{m}$, we know $\gamma>\widetilde\nu_0(1)$. Let $b=\frac{b_{1}(t)}{b_{2}(t)}\in F\setminus\{0\}$ and $k=\nu_0(b)$. Then
$$c=\frac{b}{(t^{-1})^k}=\frac{b_{1}(t)t^{\nu_0(b)}}{b_{2}(t)}$$
is a unit of $R$ and thus of $\mathcal{O}$. By (ii), $\widetilde\nu(c)=\widetilde\nu(1)$. As a result,
$$
\widetilde\nu(b)=\widetilde\nu(c(t^{-1})^k)=\widetilde\nu(c)\bm{+}\widetilde\nu((t^{-1})^k)=\widetilde\nu((t^{-1})^k)=k\gamma.
$$

Fix $x\in K^\times$, we can fix a nonzero polynomial such that
$$
\sum_{j=0}^{p}a_jx^j=0,\text{ each } a_j\in F,
$$
with at least two nonzero terms. For each $a_{j}\neq0$,
$$
\widetilde\nu(a_jx^{j})=\nu_{0}(a_j)\gamma\bm{+}j\widetilde\nu(x).$$

Set $M=\{\widetilde\nu(a_jx^{j})\mid a_{j}\neq 0, 0\leq j\leq n\}$. $M$ contains a least element as $M$ is finite. Assume that $\widetilde\nu(a_rx^{r})$ is the only least element. In other words, $\widetilde\nu(a_jx^{j})>\widetilde\nu(a_rx^{r})$ for each $a_{j}\neq 0$. Now, $\widetilde\nu(\frac{a_jx^{j}}{a_rx^{r}})>\widetilde\nu(1)$ and by (ii) $\frac{a_jx^{j}}{a_rx^{r}}\in\mathfrak{m}$. In addition, $1+\sum\limits_{j\neq r,a_{j}\neq0}\frac{a_jx^{j}}{a_rx^{r}}=0$. Then we deduce that $$1=-\sum\limits_{j\neq r,a_{j}\neq0}\frac{a_jx^{j}}{a_rx^{r}}\in\mathfrak{m}.$$ It is a contradiction. As a consequence, there are different $\widetilde\nu(a_ux^{u})$ and $\widetilde\nu(a_vx^{v})$ being the least element of $M$. So we obtain $\widetilde\nu(a_ux^{u})=\widetilde\nu(a_vx^{v})$ and thus
$$(u-v)\widetilde\nu(x)=(\nu_{0}(a_v)-\nu_{0}(a_u))\gamma.$$

Thus every $\alpha\in\Gamma$, we have
$$n\alpha=m\gamma\text{ for some }n\in\mathbb{N},\ m\in\mathbb{Z}.$$
Define a mapping $\xi:\Gamma\longrightarrow\mathbb{Q}$ by $\xi(\alpha)=m/n$ if $n\alpha=m\gamma$. By Lemma 4.6, $\xi$ is well-defined. One can check that $\xi$ is an order-preserving group embedding with $\xi(\gamma)=1$.

Set
$$
\nu(x)=\xi(\widetilde\nu(x))\quad(x\ne0),
\nu(0)=+\infty\text{ and }\delta=-\nu.
$$
Then $\nu|_{F\setminus\{0\}}=\nu_0$ and further
$$\mathcal{O}=\{x\in K\mid \delta(x)\leq0\},\ \mathfrak{m}=\{x\in K\mid \delta(x)<0\}.$$
The mapping $\delta$ we desired is obtained. \end{proof}

\begin{lemma}
{\rm Let $G=\mathcal{O}/\mathfrak{m}$ and $\pi:\mathcal{O}\longrightarrow G$.
For $n\in\mathbb{N}_0$, define
$K_n=\{x\in K\mid\delta(x)\leq n\}$ and
$$\ell_n:K_n\longrightarrow G,\ \ell_n(x)=\pi(\frac{x}{t^n}).
$$
Then $K_n$ is a $\mathbb{Q}$-vector space and $\ell_n$ is $\mathbb{Q}$-linear. In addition, the following conditions hold:

(1) $\ell_n(x)=0$ if and only if $\delta(x)<n(x\in K_n)$;

(2) $\ell_n(t^n)=1$;

(3) $\ell_{r+s}(xy)=\ell_r(x)\ell_s(y)(x\in K_r,\ y\in K_s)$.}
\end{lemma}

\begin{proof}

By the proof of Proposition 4.7, $\mathcal{O}=\{x\in K:\delta(x)\leq 0\}$, $\mathfrak{m}=\{x\in K:\delta(x)<0\}$. $\pi:\mathcal{O}\to G$ is the quotient map.

For every $q\in\mathbb{Q}\setminus\{0\}$, $\delta(q)=\delta(q^{-1})=0$. So $q,q^{-1}\in\mathcal{O}$ and thus $q\in\mathcal{O}^{\times}$, $q\notin\mathfrak{m}$. This leads to $\mathbb{Q}\subseteq\mathcal{O}$ and $\mathbb{Q}\cap\mathfrak{m}=\{0\}$. Thus the field homomorphism $\pi|_{\mathbb{Q}}:\mathbb{Q}\to G$ is injective. Next, we identify $\mathbb{Q}$ with $\pi|_{\mathbb{Q}}(\mathbb{Q})$. So $\pi(q)=q$ for every $q\in\mathbb{Q}$.

Fix $n\in\mathbb{N}_0$, $0\in K_n$ because $\delta(0)=-\infty\leq n$. Suppose $x,y\in K_n$. Then $\delta(x+y)\leq\max\{\delta(x),\delta(y)\}\leq n$. So we have $x+y\in K_n$. In addition, if $q\in\mathbb{Q}\setminus\{0\}$ and $x\in K_n\setminus\{0\}$, then $\delta(qx)=\delta(q)+\delta(x)=\delta(x)\leq n$. In conclusion, $K_n$ is a $\mathbb{Q}$-vector space.

For every nonzero $x\in K_n$, it follows form $$\delta(\frac{x}{t^n})=\delta(x)-\delta(t^n)=\delta(x)-n\leq 0$$ that $\frac{x}{t^n}\in\mathcal{O}$. Thus the mapping $\ell_n$ is well-defined.

Let $x,y\in K_n$ and $a,b\in\mathbb{Q}$. Since $\pi$ is a ring homomorphism, we obtain
$$\ell_n(ax+by)=\pi(a(\frac{x}{t^n})+b(\frac{y}{t^n}))=\pi(a)\pi(\frac{x}{t^n})+\pi(b)\pi(\frac{y}{t^n})=a\ell_n(x)+b\ell_n(y).$$
Thus $\ell_n$ is a $\mathbb{Q}$-linear mapping.

For every $x\in K_n\setminus\{0\}$, we have that
$$\ell_n(x)=0\iff \frac{x}{t^n}\in\mathfrak{m}\iff\delta(\frac{x}{t^n})<0\iff\delta(x)<n.$$
In addition, $\ell_n(0)=0$ and $\delta(0)=-\infty<n$. Besides, $t^n\in K_n$ and
$$\ell_n(t^n)=\pi(\frac{t^n}{t^n})=\pi(1)=1.$$

Let $x\in K_r$, and $y\in K_s$. If $x=0$ or $y=0$, then $xy\in K_{r+s}$ and hence $\ell_{r+s}(xy)=\ell_r(x)\ell_s(y)$. Otherwise, $\delta(xy)=\delta(x)+\delta(y)\leq r+s$. This means that $xy\in K_{r+s}$. As $\frac{x}{t^r},\frac{y}{t^s}\in\mathcal{O}$ and $\pi$ is a field homomorphism, we konw
$$\ell_{r+s}(xy)=\pi(\frac{xy}{t^{r+s}})=\pi(\frac{x}{t^r}\frac{y}{t^s})=\pi(\frac{x}{t^r})\pi(\frac{y}{t^s})=\ell_r(x)\ell_s(y).$$
\end{proof}

\section{Construction of a positive multiplicative basis}

\begin{theorem}
{\rm There are finite sets $B_n\subseteq K$ for $n\in\mathbb{N}_0$ such that

(1) $B_0=\{1\},B_1=\{t,t-i\}$;

(2) $B=\bigcup\limits_{n=0}\limits^{\infty}B_n$ is a $\mathbb{Q}$-basis of $K$;

(3) $\delta(b)=n, \ell_n(b)=1 (b\in B_n)$;

(4) $b,c\in B$ implies $bc\in P_B$.}

\end{theorem}

\begin{proof}

In the whole proof, we identify $\mathbb{Q}$ with $\pi|_{\mathbb{Q}}(\mathbb{Q})$ as in Lemma 4.8. By Lemma 4.7, we have $$\delta(i)=\frac{1}{2}\delta(i^2)=\frac{1}{2}\delta(-1)=0.$$ By Lemma 4.3, we can enumerate $K=\{x_1,x_2,\ldots\}$ with $x_1=i$.

We can choose a strictly increasing sequence of positive integers $(N_j)_{j\geq 1}$ such that $N_1=1$, $N_2\geq 3$ and $N_j>\delta(x_j)$ for every $j\geq 1$. Indeed, take $N_1=1$, $N_{j}=j+1$($2\leq j$),$x_1=i$ and $x_{j}=t^{j-1}$ ($2\leq j$). Then we have $$j+1=N_j>\delta(x_j)=j-1$$ for every $j\geq 2$. So this sequence is given.

Put $B_0=\{1\}$, $V_0=\mathbb{Q}$, and $P_0=\mathbb{Q}_{\geq 0}$. Then $B_0$ is a basis of $V_0$ and $V_0\subseteq K_0$. Besides, $\delta(1)=0$ and $\ell_0(1)=1$.

Take
$A_1=\{t,t-i\}$,
$B_1=A_1$,
and
$$V_1=V_0+\operatorname{span}_{\mathbb{Q}}A_1.$$
Then $B_{\leq 1}=\{1,t,t-i\}$ is a $\mathbb{Q}$-linearly independent set. So $t+V_0$ and $t-i+V_0$ are linearly independent in $V_1/V_0$ and further $B_{\leq 1}$ forms a basis of $V_1$.

By Lemma 4.7 and Lemma 4.8, $\delta(t)=1$,$\ell_1(t)=1$ and
$\delta(t-i)\leq\max\{\delta(t),\delta(i)\}=1$. This means that $t-i\in K_1$. It follows form $\delta(i)=0$ that $\ell_1(i)=0$. Therefore
$\ell_1(t-i)=\ell_1(t)-\ell_1(i)=1$.
Again by Lemma~4.8, the facts that $t-i\in K_1$ and $\ell_1(t-i)\neq 0$ imply $\delta(t-i)=1$. Since $D_1$ is a $\mathbb{Q}$-vector space containing $1$, $t$, and $t-i$, we also have $V_1\subseteq K_1$.

Kefine $P_1=P_{B_{\leq 1}}$. Clearly $A_1\subseteq P_1$. We also have $P_1\cap V_0=P_0$. To see this, let $p\in P_1\cap V_0$. Then
$p=a+bt+c(t-i)$
for some $a,b,c\in\mathbb{Q}_{\geq 0}$. Since $p\in V_0=\mathbb{Q}$, we have $\ell_1(p)=0$. Applying $\ell_1$ to the displayed expression gives $0=b+c$. The nonnegativity of $b$ and $c$ forces $b=c=0$, so $p=a\in P_0$. Conversely, every element of $P_0$ belongs to $P_1\cap V_0$. Thus the required properties hold for the first layer.

Now let $n\geq 2$, and suppose that the finite sets $B_0,\ldots,B_{n-1}$ have already been chosen. Assume that $B_{\leq n-1}$ is a basis of
$V_{n-1}=\operatorname{span}_{\mathbb{Q}}B_{\leq n-1}$,
and that $\delta(b)=r$ and $\ell_r(b)=1$ for every $b\in B_r$ with $0\leq r<n$. Write $P_{n-1}=P_{B_{\leq n-1}}$.

Every $v\in V_{n-1}$ satisfies $\delta(v)\leq n-1$. Indeed, this is immediate when $v=0$. Otherwise, express $v$ as a finite rational linear combination of elements of $B_{\leq n-1}$. Each nonzero summand has $\delta$-value at most $n-1$, since multiplication by a nonzero rational number does not change the $\delta$-value. Repeated application of the degree inequality then gives $\delta(v)\leq n-1$. In particular,
$\delta(v)<n$
and hence
$\ell_n(v)=0$
for every $v\in V_{n-1}$.

Define
$$A_n=\{t^n\}\cup\{bc:b\in B_r,\ c\in B_s,\ r,s\geq 1,\ r+s=n\}\cup\{t^n-x_j:N_j=n\},$$
where $r$ and $s$ in the middle set range over positive integers.

The set $A_n$ is nonempty because it contains $t^n$. It is finite: there are only finitely many pairs of positive integers $(r,s)$ with $r+s=n$, and the sets $B_r$ and $B_s$ are finite. Furthermore, $r,s<n$ for every such pair, so these sets have already been constructed. Finally, the strict increase of the sequence $(N_j)$ implies that there is at most one index $j$ satisfying $N_j=n$. Thus the definition of $A_n$ uses only previously chosen data and produces a finite set.

We claim that every $a\in A_n$ satisfies $\delta(a)=n$ and $\ell_n(a)=1$.

For $a=t^n$, these identities follow from $\delta(t)=1$ and Lemma~4.8.

If $a=bc$, where $b\in B_r$, $c\in B_s$, and $r+s=n$, then the induction hypothesis gives
$\delta(bc)=\delta(b)+\delta(c)=r+s=n$.
The multiplicative identity in Lemma~4.8 also gives
$\ell_n(bc)=\ell_r(b)\ell_s(c)=1$.

Finally, suppose that $a=t^n-x_j$ with $N_j=n$. By the choice of $N_j$, we have $\delta(x_j)<n$. Therefore
$$\delta(t^n-x_j)\leq\max\{\delta(t^n),\delta(x_j)\}=n,$$
so $t^n-x_j\in K_n$. Moreover, Lemma~4.8 gives $\ell_n(x_j)=0$, including when $x_j=0$. Consequently,
$$\ell_n(t^n-x_j)=\ell_n(t^n)-\ell_n(x_j)=1.$$
Since $\ell_n(t^n-x_j)\neq 0$, the kernel characterization in Lemma~4.8 shows that $\delta(t^n-x_j)$ is not strictly less than $n$. Combined with the preceding upper bound, this yields $\delta(t^n-x_j)=n$. This proves the claim.

Set
$W_n=V_{n-1}+\operatorname{span}_{\mathbb{Q}}A_n$.
Both $V_{n-1}$ and $A_n$ are contained in $K_n$. Since $K_n$ is a $\mathbb{Q}$-vector space, it follows that $W_n\subseteq K_n$.

Although $\ell_n$ is defined with values in $k$, its restriction to $W_n$ takes values in the canonical copy of $\mathbb{Q}$. Indeed, let $w\in W_n$ and write
$w=v+\sum_{a\in A_n}q_a a$,
where $v\in V_{n-1}$ and $q_a\in\mathbb{Q}$. Using $\ell_n(v)=0$ and $\ell_n(a)=1$ for all $a\in A_n$, we obtain
$\ell_n(w)=\sum_{a\in A_n}q_a\,1_k\in\mathbb{Q}\cdot 1_k$.
Identifying $\mathbb{Q}\cdot 1_k$ with $\mathbb{Q}$, we may therefore regard the restriction of $\ell_n$ as a map
$\lambda_n=\ell_n|_{W_n}:W_n\to\mathbb{Q}$.
This map is well defined because it is the restriction of the already defined map $\ell_n$. The displayed expression for $w$ need not be unique; every such expression gives the same value because its coefficient sum equals $\ell_n(w)$.

The map $\lambda_n$ is $\mathbb{Q}$-linear and satisfies
$\lambda_n(V_{n-1})=\{0\}$
and
$\lambda_n(a)=1$
for every $a\in A_n$.

We now apply the cone extension lemma, Lemma~3.2, with the old space $V_{n-1}$, its basis $B_{\leq n-1}$, the finite nonempty set $A_n$, the enlarged space $W_n$, and the linear map $\lambda_n$. Its hypotheses hold because
$W_n=V_{n-1}+\operatorname{span}_{\mathbb{Q}}A_n$,
$\lambda_n(V_{n-1})=\{0\}$, and $\lambda_n(A_n)=\{1\}$.

The lemma yields a finite set $B_n\subseteq W_n$ such that, on setting
$V_n=W_n$,
$B_{\leq n}=B_{\leq n-1}\cup B_n$,
and
$P_n=P_{B_{\leq n}}$,
the set $B_{\leq n}$ is a basis of $V_n$, the inclusions and equalities
$A_n\subseteq P_n$
and
$P_n\cap V_{n-1}=P_{n-1}$
hold, and $\lambda_n(b)=1$ for every $b\in B_n$.

Since $B_n\subseteq V_n=W_n\subseteq K_n$, each $b\in B_n$ satisfies $\delta(b)\leq n$. On the other hand, $\ell_n(b)=\lambda_n(b)=1\neq 0$, so Lemma~4.8 implies that $\delta(b)$ is not strictly less than $n$. Hence $\delta(b)=n$ for every $b\in B_n$. We have also obtained $V_n\subseteq K_n$ and $V_{n-1}\subseteq V_n$. This completes the induction.

We have thus constructed finite sets $B_n$ for all $n\in\mathbb{N}_0$, with $B_0=\{1\}$ and $B_1=\{t,t-i\}$. For every $n$, the set $B_{\leq n}$ is a basis of $V_n$, and every $b\in B_n$ satisfies $\delta(b)=n$ and $\ell_n(b)=1$. Moreover, $A_n\subseteq P_n$ for every $n\geq 1$.

Set $B=\bigcup_{n\geq 0}B_n$. We first prove that $B$ is linearly independent over $\mathbb{Q}$. Consider a finite relation
$\sum_{j=1}^{m}q_jb_j=0$,
where the $b_j$ are distinct elements of $B$ and $q_j\in\mathbb{Q}$. Since only finitely many elements occur, there exists $N\in\mathbb{N}_0$ such that $b_1,\ldots,b_m\in B_{\leq N}$. The set $B_{\leq N}$ is linearly independent, so $q_1=\cdots=q_m=0$. Thus $B$ is linearly independent.

We next prove that $B$ spans $K$. Fix $j\geq 1$ and set $n=N_j$. By construction, both $t^n$ and $t^n-x_j$ belong to $A_n$. For $j=1$, this is the identity $A_1=\{t,t-i\}$; for all other $j$, it follows from the definition of $A_n$. Therefore both elements belong to $V_n$, and
$x_j=t^n-(t^n-x_j)\in V_n\subseteq\operatorname{span}_{\mathbb{Q}}B$.
Since the enumeration $(x_j)$ contains every element of $K$, we obtain
$K=\operatorname{span}_{\mathbb{Q}}B$.
Together with linear independence, this proves that $B$ is a $\mathbb{Q}$-basis of $K$.

Finally, we verify the required positivity of products. For each $n$, the inclusion $B_{\leq n}\subseteq B$ gives $P_n\subseteq P_B$. Let $b,c\in B$, and choose $r,s\in\mathbb{N}_0$ such that $b\in B_r$ and $c\in B_s$.

If $r=0$, then $b=1$ because $B_0=\{1\}$, and hence $bc=c\in P_B$. Similarly, if $s=0$, then $bc=b\in P_B$.

Suppose now that $r,s\geq 1$, and set $n=r+s$. Then $r,s<n$, so both $B_r$ and $B_s$ had already been constructed when $A_n$ was defined. Consequently, the product $bc$ belongs to the middle set in the definition of $A_n$. Hence
$bc\in A_n\subseteq P_n\subseteq P_B$,
as required.

\end{proof}

\begin{corollary}
{\rm For every $n\in\mathbb{N}_0$,
$$
P_B\cap V_n=P_n.
$$
More generally, the $B$-expansion of an element of $V_n$ is its $B_{\leq n}$-expansion with all remaining coordinates zero.}
\end{corollary}

\begin{proof}
The set $B$ extends the basis $B_{\leq n}$, so uniqueness of basis coordinates gives the assertion.
\end{proof}

\section{The coordinate lattice order on the countable field}

\begin{proposition}
{\rm Let $B$ be as in Theorem 5.1. Define
$$
P_K=P_B\text{ and }x\leq_K y\text{ if and only if }y-x\in P_K.
$$
Then $\leq_K$ is a compatible lattice order on the original field $K$. For
$$
x=\sum_{b\in B}x_b b,\qquad y=\sum_{b\in B}y_b b,
$$
we have

(1) $x\leq_K y$ if and only if $x_b\leq y_b \quad(b\in B)$;

(2) $x\vee_K y=\sum_{b\in B}\max\{x_b,y_b\}\,b$;

(3) $x\wedge_K y=\sum_{b\in B}\min\{x_b,y_b\}\,b$.

The restriction of $\leq_K$ to $\mathbb{Q}$ is its usual order and $0<_K1$.}
\end{proposition}

\begin{proof}
By Theorem~5.1, the set $B$ is a $\mathbb{Q}$-basis of $K$, the element $1$ belongs to $B$, and $bc\in P_B$ for every $b,c\in B$. Put $P=P_K=P_B$.

Every $x\in K$ has a unique expansion
$x=\sum_{b\in B}x_b b$,
where $x_b\in\mathbb{Q}$ and only finitely many coefficients are nonzero. We write
$\operatorname{supp}_B(x)=\{b\in B:x_b\neq 0\}$.
Thus $\operatorname{supp}_B(x)$ is finite. All sums indexed by $B$ below are understood in this finite-support sense.

By definition,
$$P=\{\sum_{b\in S}q_b b:S\subseteq B\text{ is finite and }q_b\in\mathbb{Q}_{\geq 0}\}.$$
The uniqueness of the $B$-expansion implies that, for every $x\in K$,
$x\in P$ if and only if $x_b\geq 0$ for every $b\in B$.

We first verify the properties of $P$ needed to define an order. The zero element belongs to $P$, since all its coordinates are zero. If $p,q\in P$, then
$p+q=\sum_{b\in B}(p_b+q_b)b$.
Every coefficient $p_b+q_b$ is nonnegative, and the support of $p+q$ is contained in the finite set $\operatorname{supp}_B(p)\cup\operatorname{supp}_B(q)$. Hence $p+q\in P$. Thus $P+P\subseteq P$.

Similarly, if $a\in\mathbb{Q}_{\geq 0}$ and $p\in P$, then
$ap=\sum_{b\in B}(ap_b)b$.
All coefficients are nonnegative, so $ap\in P$. Therefore $\mathbb{Q}_{\geq 0}P\subseteq P$. In particular, $P$ is closed under finite nonnegative rational linear combinations.

We also have $P\cap(-P)=\{0\}$. Indeed, suppose that $p\in P$ and $-p\in P$. Write
$p=\sum_{b\in B}p_b b$.
Since $p\in P$, every $p_b$ is nonnegative. On the other hand,
$-p=\sum_{b\in B}(-p_b)b$
is the unique $B$-expansion of $-p$. Since $-p\in P$, every $-p_b$ is also nonnegative. Consequently, $p_b=0$ for all $b\in B$, and hence $p=0$. The reverse inclusion follows from $0\in P$.

Define the relation $\leq_K$ on $K$ by
$x\leq_K y\iff y-x\in P$.
We verify that this is a partial order.

For every $x\in K$, we have $x-x=0\in P$, so $x\leq_K x$. Thus the relation is reflexive.

Suppose that $x\leq_K y$ and $y\leq_K x$. Then $y-x\in P$ and $x-y=-(y-x)\in P$. Hence $y-x\in P\cap(-P)=\{0\}$, which implies $x=y$. Thus the relation is antisymmetric.

Suppose that $x\leq_K y$ and $y\leq_K z$. Then $y-x\in P$ and $z-y\in P$. Since $P$ is closed under addition,
$z-x=(z-y)+(y-x)\in P$.
Therefore $x\leq_K z$, proving transitivity. Hence $\leq_K$ is a partial order.

This order is compatible with the original addition of $K$. Indeed, for any $x,y,z\in K$, the identity
$(y+z)-(x+z)=y-x$
gives
$x\leq_K y\iff x+z\leq_K y+z$.
Moreover, if $x\leq_K y$ and $a\in\mathbb{Q}_{\geq 0}$, then
$ay-ax=a(y-x)\in P$,
so $ax\leq_K ay$.

We next verify compatibility with multiplication. Let $x,y\in P$. There are finite sets $S,T\subseteq B$ and nonnegative rational numbers $q_b$ and $r_c$ such that
$x=\sum_{b\in S}q_b b$
and
$y=\sum_{c\in T}r_c c$.
Using the original multiplication of $K$ and distributivity, we obtain
$xy=\sum_{b\in S}\sum_{c\in T}q_b r_c\,bc$.
For every pair $(b,c)\in S\times T$, Theorem~5.1 gives $bc\in P$. Since $q_b r_c\geq 0$, closure under nonnegative rational scalar multiplication yields $q_b r_c\,bc\in P$. There are only finitely many such pairs, so closure under addition gives $xy\in P$. Thus $PP\subseteq P$.

Consequently, $0\leq_K x$ and $0\leq_K y$ imply $0\leq_K xy$. More generally, suppose that $x\leq_K y$ and $0\leq_K a$. Then $y-x\in P$ and $a\in P$, so
$ay-ax=a(y-x)\in P$.
Hence $ax\leq_K ay$. Since $K$ is commutative, we also have $xa\leq_K ya$. This proves the required compatibility with multiplication.

We now establish the coordinate characterization of the order. Let
$x=\sum_{b\in B}x_b b$
and
$y=\sum_{b\in B}y_b b$.
Then
$y-x=\sum_{b\in B}(y_b-x_b)b$.
By uniqueness of the $B$-expansion and the coordinate description of $P$,
$$x\leq_K y\iff y-x\in P\iff y_b-x_b\geq 0\text{ for every }b\in B.$$
Equivalently,
$x\leq_K y\iff x_b\leq y_b\text{ for every }b\in B$.

To prove that this partial order is a lattice order, fix $x,y\in K$ and define
$u=\sum_{b\in B}\max\{x_b,y_b\}b$.
This expression defines an element of $K$. Indeed, if
$b\notin\operatorname{supp}_B(x)\cup\operatorname{supp}_B(y)$,
then $x_b=y_b=0$, and hence $\max\{x_b,y_b\}=0$. Thus the support of $u$ is contained in the finite set $\operatorname{supp}_B(x)\cup\operatorname{supp}_B(y)$.

For every $b\in B$, we have
$x_b\leq\max\{x_b,y_b\}=u_b$
and
$y_b\leq\max\{x_b,y_b\}=u_b$.
The coordinate characterization of the order therefore gives $x\leq_K u$ and $y\leq_K u$. Hence $u$ is an upper bound of $x$ and $y$.

Let $v\in K$ be any other upper bound of $x$ and $y$. Then $x\leq_K v$ and $y\leq_K v$, so $x_b\leq v_b$ and $y_b\leq v_b$ for every $b\in B$. Consequently,
$u_b=\max\{x_b,y_b\}\leq v_b$
for every $b\in B$. Applying the coordinate characterization once more, we obtain $u\leq_K v$. Thus $u$ is the least upper bound of $x$ and $y$, proving
$x\vee_K y=\sum_{b\in B}\max\{x_b,y_b\}b$.

Similarly, define
$w=\sum_{b\in B}\min\{x_b,y_b\}b$.
Its support is also contained in $\operatorname{supp}_B(x)\cup\operatorname{supp}_B(y)$, so $w\in K$. For every $b\in B$, we have $w_b\leq x_b$ and $w_b\leq y_b$. Therefore $w\leq_K x$ and $w\leq_K y$, so $w$ is a lower bound of $x$ and $y$.

If $v\in K$ is any other lower bound, then $v\leq_K x$ and $v\leq_K y$. Hence $v_b\leq x_b$ and $v_b\leq y_b$ for every $b\in B$, which implies
$v_b\leq\min\{x_b,y_b\}=w_b$
for every $b\in B$. Thus $v\leq_K w$. This proves that $w$ is the greatest lower bound of $x$ and $y$, and therefore
$x\wedge_K y=\sum_{b\in B}\min\{x_b,y_b\}b$.

Equivalently, the meet formula follows from
$x\wedge_K y=-\bigl((-x)\vee_K(-y)\bigr)$,
because
$-\max\{-x_b,-y_b\}=\min\{x_b,y_b\}$
for every $b\in B$. Since every pair of elements has both a least upper bound and a greatest lower bound, $\leq_K$ is a lattice order. Together with the compatibility properties already proved, this makes the original ring structure of $K$ a lattice-ordered ring.

Finally, we check the restriction of the order to $\mathbb{Q}$. Since $1\in B$, every $q\in\mathbb{Q}$ has the $B$-expansion $q=q\cdot 1$, with coefficient zero at every basis element other than $1$. By uniqueness, these are its coordinates. Therefore $q\in P$ if and only if $q\geq 0$ in the usual order on $\mathbb{Q}$. Hence
$P_K\cap\mathbb{Q}=\mathbb{Q}_{\geq 0}$.

For $p,q\in\mathbb{Q}$, it follows that
$$p\leq_K q\iff q-p\in P_K\cap\mathbb{Q}\iff q-p\geq 0\iff p\leq q,$$
where the final inequality is taken in the usual order on $\mathbb{Q}$. Thus the restriction of $\leq_K$ to $\mathbb{Q}$ is its usual order. In particular, $1\in P_K$ and $1\neq 0$, so $0<_K1$.
\end{proof}

\begin{proposition}
{\rm The choices in Theorem 5.1 can be made so that
$$
B_1=\{t,t-i\},\ B_2=\{u,v\}\text{ and }
u=t^2+it-1, v=(t-i)^2=t^2-2it-1.$$
For these choices,
$$
t^2=1+\frac23u+\frac13v,\ t(t-i)=1+\frac13u+\frac23v\text{ and }(t-i)^2=v.$$
Moreover, in the final order on $K$,
$$
0\vee_K i=t,\ 0\wedge_K i=i-t\text{ and }1\wedge_K t=0.$$}
\end{proposition}

\begin{proof}

We use the notation and recursive construction from Theorem~5.1. Choose the first layer as $B_1=\{t,t-i\}$, and retain the choices $N_1=1$ and $N_2\geq 3$ for the strictly increasing sequence $(N_j)_{j\geq 1}$. The space obtained after the first stage is
$V_1=\operatorname{span}_{\mathbb{Q}}\{1,t,t-i\}=\operatorname{span}_{\mathbb{Q}}\{1,t,i\}$.

We first determine the target set at the second stage. Since $N_1=1$ and $N_j\geq N_2\geq 3$ for every $j\geq 2$, there is no index $j$ such that $N_j=2$. Thus no element of the form $t^2-x_j$ occurs at this stage. Moreover, the only positive integers $r,s$ satisfying $r+s=2$ are $r=s=1$. Consequently,
$A_2=\{t^2\}\cup\{bc:b,c\in B_1\}=\{t^2,t(t-i),(t-i)^2\}$.

Put $e_1=t^2$, $e_2=t(t-i)=t^2-it$, and $a_3=(t-i)^2=t^2-2it-1$. Then
$a_3=-1-e_1+2e_2$.
Indeed,
$$-1-e_1+2e_2=-1-t^2+2(t^2-it)=t^2-2it-1.$$
It follows that
$W_2=V_1+\operatorname{span}_{\mathbb{Q}}A_2=V_1+\operatorname{span}_{\mathbb{Q}}\{e_1,e_2\}$.

We claim that $e_1+V_1$ and $e_2+V_1$ are linearly independent in $W_2/V_1$. First, note that $t$ is transcendental over $\mathbb{Q}(i)$. Otherwise, since $[\mathbb{Q}(i):\mathbb{Q}]=2$, the tower formula would give
$$[\mathbb{Q}(i,t):\mathbb{Q}]=[\mathbb{Q}(i,t):\mathbb{Q}(i)]\,[\mathbb{Q}(i):\mathbb{Q}]<\infty,$$
forcing $t$ to be algebraic over $\mathbb{Q}$, contrary to its choice.

Now suppose that $\alpha e_1+\beta e_2\in V_1$ for some $\alpha,\beta\in\mathbb{Q}$. There exist $c_0,c_1,c_2\in\mathbb{Q}$ such that
$\alpha e_1+\beta e_2=c_0+c_1t+c_2i$.
Substitution of the expressions for $e_1$ and $e_2$ yields
$$(\alpha+\beta)t^2+(-\beta i-c_1)t-(c_0+c_2i)=0.$$
Since $t$ is transcendental over $\mathbb{Q}(i)$, all coefficients of this polynomial must vanish. Hence
$\alpha+\beta=0$, $-\beta i-c_1=0$, and $c_0+c_2i=0$.
Comparing real and imaginary parts in the last two equalities gives $\beta=c_1=0$ and $c_0=c_2=0$. The first equality then gives $\alpha=0$. This proves the claim. Since these two cosets also span $W_2/V_1$, they form a basis of that quotient.

We now make the choice in the cone extension lemma explicit. Relative to the quotient basis $e_1+V_1,e_2+V_1$, the three elements of $A_2$ have coordinate pairs $(1,0)$, $(0,1)$, and $(-1,2)$, respectively. Their corresponding components in $V_1$ are $0$, $0$, and $-1$.

Thus we may take $M=1$: adding $M$ to each of these quotient coordinates gives the nonnegative pairs $(2,1)$, $(1,2)$, and $(0,3)$. We may also take $h=1\in P_1$, since the old components become $1$, $1$, and $0$ after addition of $h$, and all three belong to $P_1$.

The quotient dimension is $d=2$, so the constant in the construction is $D=1+dM=3$. The two new elements prescribed by the cone extension construction are therefore
$u=De_1-M(e_1+e_2)-h=2e_1-e_2-1$
and
$v=De_2-M(e_1+e_2)-h=2e_2-e_1-1$.
Expanding these expressions in $K$, we obtain
$$u=2t^2-(t^2-it)-1=t^2+it-1$$
and
$$v=2(t^2-it)-t^2-1=t^2-2it-1=(t-i)^2.$$

For completeness, we verify directly that $B_2=\{u,v\}$ satisfies all requirements of the second stage.

Modulo $V_1$, the defining equations become
$$u+V_1=2(e_1+V_1)-(e_2+V_1)$$
and
$$v+V_1=-(e_1+V_1)+2(e_2+V_1).$$
The corresponding change-of-basis matrix is
$\begin{pmatrix}2&-1\\-1&2\end{pmatrix}$,
whose determinant is $3\neq 0$. Hence $u+V_1$ and $v+V_1$ also form a basis of $W_2/V_1$. Since $\{1,t,t-i\}$ is a basis of $V_1$, the set $\{1,t,t-i,u,v\}$ is a basis of $W_2$. We therefore set $V_2=W_2$ and $B_2=\{u,v\}$.

We next check the degree and normalization conditions. Since $\delta(t)=\delta(t-i)=1$, multiplicativity gives
$\delta(e_1)=\delta(t^2)=2$,
$\delta(e_2)=\delta(t(t-i))=2$,
and
$\delta(a_3)=\delta((t-i)^2)=2$.
Also, $V_1\subseteq K_1\subseteq K_2$. As $D_2$ is a $\mathbb{Q}$-vector space, it follows that
$W_2=V_1+\operatorname{span}_{\mathbb{Q}}A_2\subseteq K_2$.
In particular, $u,v\in K_2$.

By Lemma~4.8,
$\ell_2(e_1)=1$
and
$\ell_2(e_2)=\ell_1(t)\ell_1(t-i)=1$.
Since $\delta(1)=0<2$, the same lemma gives $\ell_2(1)=0$. Using linearity, we conclude that
$$\ell_2(u)=2\ell_2(e_1)-\ell_2(e_2)-\ell_2(1)=1$$
and
$$\ell_2(v)=2\ell_2(e_2)-\ell_2(e_1)-\ell_2(1)=1.$$
Because $u,v\in D_2$ and their $\ell_2$-values are nonzero, the equivalence $\ell_2(x)=0\iff\delta(x)<2$ implies $\delta(u)=\delta(v)=2$.

We now derive the asserted positive expansions. From
$u=2e_1-e_2-1$
and
$v=2e_2-e_1-1$,
we obtain
$$2u+v=2(2e_1-e_2-1)+(2e_2-e_1-1)=3e_1-3.$$
Therefore
$t^2=e_1=1+\frac{2}{3}u+\frac{1}{3}v$.
Similarly,
$$u+2v=(2e_1-e_2-1)+2(2e_2-e_1-1)=3e_2-3,$$
and hence
$t(t-i)=e_2=1+\frac{1}{3}u+\frac{2}{3}v$.
Finally, the definition of $v$ gives $(t-i)^2=v$.

Let $P_2=P_{\{1,t,t-i,u,v\}}$. Every coefficient in these three expansions is a nonnegative rational number. Consequently, $A_2\subseteq P_2$.

We also verify the contraction condition $P_2\cap V_1=P_1$. Since the old basis is contained in the new basis, we immediately have $P_1\subseteq P_2\cap V_1$. Conversely, let $z\in P_2\cap V_1$. By the definition of $P_2$, we can write
$z=p+\alpha u+\beta v$,
where $p\in P_1$ and $\alpha,\beta\in\mathbb{Q}_{\geq 0}$. Every element of $V_1$ has $\delta$-value at most $1$, so $\ell_2$ vanishes on $V_1$. Applying $\ell_2$ to this expression, we obtain
$$0=\ell_2(z)=\ell_2(p)+\alpha\ell_2(u)+\beta\ell_2(v)=\alpha+\beta.$$
Since $\alpha$ and $\beta$ are nonnegative rational numbers, both must be zero. Thus $z=p\in P_1$, proving the reverse inclusion.

We have checked that the prescribed second layer gives a basis extension, that $\delta(b)=2$ and $\ell_2(b)=1$ for every $b\in B_2$, that $A_2\subseteq P_2$, and that $P_2\cap V_1=P_1$. Thus this choice satisfies all conditions of the recursive construction. Starting from these first two layers, we continue the construction of Theorem~5.1 for every stage $n\geq 3$. This produces a final basis $B$ with the required properties.

It remains to establish the three lattice identities in the final order on $K$. By Corollary~5.2, the final $B$-expansion of any element of $V_2$ is its $B_{\leq 2}$-expansion with zero coefficients at all later basis elements. Indeed, $B$ contains $B_{\leq 2}$, so adjoining zero coefficients gives a $B$-expansion; uniqueness of the final basis expansion shows that these are the actual coordinates. Thus the coordinate calculations below remain valid after every later stage.

In the final basis, we have
$i=t-(t-i)$.
Accordingly, the coordinate of $i$ at the basis element $t$ is $1$, its coordinate at the basis element $t-i$ is $-1$, and all its other coordinates are zero. The zero element has every coordinate equal to zero.

By the coordinate formula for joins in Proposition~6.1, we therefore obtain
$$0\vee_K i=\max\{0,1\}t+\max\{0,-1\}(t-i)=t.$$
All omitted coordinates are zero because they are maxima of two zero coordinates.

Similarly, the coordinate formula for meets gives
$$0\wedge_K i=\min\{0,1\}t+\min\{0,-1\}(t-i)=-(t-i)=i-t.$$

Finally, $1$ and $t$ are distinct elements of the final basis. The only nonzero coordinate of $1$ is its coordinate $1$ at the basis element $1$, while the only nonzero coordinate of $t$ is its coordinate $1$ at the basis element $t$. Thus the minimum of their two coordinates is zero at each basis element. Proposition~6.1 consequently gives $1\wedge_K t=0$.

\end{proof}

\begin{proposition}
{\rm Every nonzero $p\in P_K$ has nonnegative integral degree. More precisely, if
$$p=\sum_{b\in B}q_b b\ne0,\ q_b\in\mathbb{Q}_{\geq0}\text{ and }
m=\max\{n\mid \exists b\in B_n,\ q_b>0\},$$
then
$$
\delta(p)=m,\ \ell_m(p)=\sum_{b\in B_m}q_b\in\mathbb{Q}_{>0}.
$$
In particular,
$$t\in P_K,\ t^{-1}\notin P_K,\ -t^{-1}\notin P_K,i\notin P_K\text{ and } -i\notin P_K.$$}
\end{proposition}

\begin{proof}
Recall that $P_K=P_B$ is the set of all finite nonnegative rational linear combinations of elements of $B$. By Theorem~5.1, we have $B=\bigcup_{n\geq 0}B_n$, and every $b\in B_n$ satisfies $\delta(b)=n$ and $\ell_n(b)=1$. In particular, the layers are pairwise disjoint: if $b\in B_r\cap B_s$, then $r=\delta(b)=s$.

Let $p\in P_K\setminus\{0\}$. Since $B$ is a $\mathbb{Q}$-basis of $K$, the element $p$ has a unique expansion
$p=\sum_{b\in B}q_b b$,
where $q_b\in\mathbb{Q}_{\geq 0}$ and only finitely many coefficients are nonzero.

Set $S=\{b\in B:q_b>0\}$. This set is finite, and it is nonempty because $p\neq 0$. Since every element of $B$ belongs to exactly one layer, the set
$I=\{n\in\mathbb{N}_0:S\cap B_n\neq\varnothing\}$
is finite and nonempty. Its maximum therefore exists. Put
$m=\max I=\max\{n\in\mathbb{N}_0:(\exists b\in B_n)\ q_b>0\}$.
By definition, $m\in\mathbb{N}_0$, every element of $S$ belongs to a layer of index at most $m$, and at least one element of $S$ belongs to $B_m$.

We first show that $p\in K_m$. If $b\in S\cap B_r$, then $q_b$ is a nonzero rational number, so $\delta(q_b)=0$. Consequently,
$$\delta(q_b b)=\delta(q_b)+\delta(b)=r\leq m.$$
Since the expansion of $p$ is finite, repeated application of the inequality $\delta(x+y)\leq\max\{\delta(x),\delta(y)\}$ gives
$\delta(p)\leq\max_{b\in S}\delta(q_b b)\leq m$.
Thus $p\in K_m$, and $\ell_m(p)$ is well defined.

We now compute $\ell_m(p)$. Every $b\in S$ belongs to $K_m$, so the $\mathbb{Q}$-linearity of $\ell_m$ gives
$\ell_m(p)=\sum_{b\in S}q_b\ell_m(b)$.
If $b\in S\cap B_r$ with $r<m$, then $\delta(b)=r<m$. The kernel characterization in Lemma~4.8 therefore gives $\ell_m(b)=0$. If $b\in S\cap B_m$, then $\ell_m(b)=1$ by the normalization of the basis elements. Hence
$$\ell_m(p)=\left(\sum_{b\in S\cap B_m}q_b\right)1_k=\left(\sum_{b\in B_m}q_b\right)1_k.$$

Let $c=\sum_{b\in B_m}q_b\in\mathbb{Q}$. All summands are nonnegative, and at least one is strictly positive because $S\cap B_m\neq\varnothing$. Therefore $c>0$ in the usual order on $\mathbb{Q}$.

By Lemma~4.8, the canonical map $\mathbb{Q}\to k$ is injective. Thus the image $c\,1_k$ of this nonzero rational number is nonzero in $k$. Identifying $\mathbb{Q}$ with its canonical image in $k$, we obtain
$\ell_m(p)=\sum_{b\in B_m}q_b\in\mathbb{Q}_{>0}$.
Here the notation $\mathbb{Q}_{>0}$ refers to the image of the positive rational numbers under this embedding.

Since $p\in K_m$ and $\ell_m(p)\neq 0$, Lemma~4.8 implies that $\delta(p)$ is not strictly less than $m$. Together with the already established inequality $\delta(p)\leq m$, this yields
$\delta(p)=m$.
Thus every nonzero element of $P_K$ has nonnegative integral degree, and the asserted formula for $\ell_m(p)$ also holds. The argument includes $m=0$, since $\ell_0$ is defined and there are then no lower layers contributing to the sum.

We next prove the stated inclusions and exclusions. Since $t\in B_1\subseteq B$, the element $t$ is a nonnegative rational linear combination of basis elements, with coefficient $1$ at $t$ and zero elsewhere. Hence $t\in P_K$.

Because $t\neq 0$ and $tt^{-1}=1$, multiplicativity gives
$$0=\delta(1)=\delta(t)+\delta(t^{-1})=1+\delta(t^{-1}).$$
Therefore $\delta(t^{-1})=-1$. Also, $\delta(-1)=0$, since $-1$ is a nonzero rational number. It follows that
$\delta(-t^{-1})=\delta(-1)+\delta(t^{-1})=-1$.

Both $t^{-1}$ and $-t^{-1}$ are nonzero. If either belonged to $P_K$, the first part of the proof would force its degree to belong to $\mathbb{N}_0$, contradicting its degree $-1$. Consequently, $t^{-1}\notin P_K$ and $-t^{-1}\notin P_K$.

Finally, $t$ and $t-i$ are distinct elements of the basis $B$, and
$i=t-(t-i)$.
By uniqueness of the $B$-expansion, the coordinate of $i$ at $t$ is $1$, its coordinate at $t-i$ is $-1$, and all its other coordinates are zero. An element belongs to $P_B$ precisely when all its $B$-coordinates are nonnegative. The negative coordinate at $t-i$ therefore shows that $i\notin P_K$.

Similarly,
$-i=-t+(t-i)$
is the unique $B$-expansion of $-i$. Its coordinate at $t$ is $-1$, so $-i\notin P_K$.
\end{proof}

\section{An algebraically closed ultrapower of cardinality continuum}

\begin{lemma}
{\rm There exists a nonprincipal ultrafilter $\mathcal{U}$ on $\mathbb{N}$ containing every cofinite subset of $\mathbb{N}$.}
\end{lemma}

\begin{proof}
The cofinite subsets form a proper filter. The union of a chain of proper filters containing this filter is again a proper filter. Zorn's lemma gives a maximal proper extension $\mathcal{U}$.

If $A\notin\mathcal{U}$ and every $F\in\mathcal{U}$ met $A$, adjoining $A$ would generate a proper larger filter. Thus some $F\in\mathcal{U}$ satisfies $F\cap A=\varnothing$, and upward closure gives
$$
\mathbb{N}\setminus A\in\mathcal{U}.
$$
Therefore $\mathcal{U}$ is an ultrafilter. It contains no finite set, because the complement of any finite set already belongs to $\mathcal{U}$.
\end{proof}

Fix such an ultrafilter. On $K^{\mathbb{N}}$ define
$$
(x_n)\sim(y_n)
\quad\Longleftrightarrow\quad
\{n:x_n=y_n\}\in\mathcal{U},
$$
and set
$$
E=K^{\mathbb{N}}/\mathcal{U}.
$$
An equivalence class is denoted by $[(x_n)]$.

\begin{proposition}
{\rm Coordinatewise addition and multiplication make $E$ an algebraically closed field of characteristic zero. The map
$$
j:K\longrightarrow E,
\qquad j(a)=[(a,a,\ldots)]
$$
is an injective field homomorphism.}
\end{proposition}

\begin{proof}
Fix the ultrafilter $\mathcal{U}$ supplied by Lemma~7.1. Recall that $\mathbb{N}\in\mathcal{U}$, $\varnothing\notin\mathcal{U}$, and $\mathcal{U}$ is closed under finite intersections and passage to supersets. We also use the ultrafilter property that, for every $A\subseteq\mathbb{N}$, exactly one of $A$ and $\mathbb{N}\setminus A$ belongs to $\mathcal{U}$.

We first verify that the relation $\sim$ is an equivalence relation on $K^{\mathbb{N}}$. For every sequence $(x_n)$, the equality set $\{n\in\mathbb{N}:x_n=x_n\}$ is $\mathbb{N}$, which belongs to $\mathcal{U}$. Thus $(x_n)\sim(x_n)$, proving reflexivity. Symmetry follows because $\{n:x_n=y_n\}=\{n:y_n=x_n\}$.

To prove transitivity, suppose that $(x_n)\sim(y_n)$ and $(y_n)\sim(z_n)$. Then
$A=\{n:x_n=y_n\}\in\mathcal{U}$
and
$B=\{n:y_n=z_n\}\in\mathcal{U}$.
Closure under finite intersections gives $A\cap B\in\mathcal{U}$. For every $n\in A\cap B$, we have $x_n=z_n$, so
$A\cap B\subseteq\{n:x_n=z_n\}$.
By upward closure, $\{n:x_n=z_n\}\in\mathcal{U}$. Hence $(x_n)\sim(z_n)$, proving transitivity.

Consequently, the quotient set $E=K^{\mathbb{N}}/\mathcal{U}$ is well defined. We denote the equivalence class of $(x_n)$ by $[(x_n)]$.

Define addition and multiplication on $E$ by
$[(x_n)]+[(y_n)]=[(x_n+y_n)]$
and
$[(x_n)][(y_n)]=[(x_ny_n)]$.
We check that these definitions are independent of the representatives.

Suppose that $(x_n)\sim(x'_n)$ and $(y_n)\sim(y'_n)$. Then the sets
$A=\{n:x_n=x'_n\}$
and
$B=\{n:y_n=y'_n\}$
belong to $\mathcal{U}$. Therefore $A\cap B\in\mathcal{U}$. For every $n\in A\cap B$, we have both
$x_n+y_n=x'_n+y'_n$
and
$x_ny_n=x'_ny'_n$.
Thus the equality sets for the sum sequences and the product sequences both contain $A\cap B$. By upward closure, both equality sets belong to $\mathcal{U}$. This proves
$(x_n+y_n)\sim(x'_n+y'_n)$
and
$(x_ny_n)\sim(x'_ny'_n)$.
Hence addition and multiplication are well defined on $E$.

Similarly, additive inverses are given by
$-[(x_n)]=[(-x_n)]$.
Indeed, equivalent sequences have equivalent coordinatewise negatives.

Set
$0_E=[(0,0,\ldots)]$
and
$1_E=[(1,1,\ldots)]$.
The ring identities follow from the corresponding identities in $K$, applied at every coordinate. For example, if $X=[(x_n)]$, $Y=[(y_n)]$, and $Z=[(z_n)]$, then
$$(X+Y)+Z=[((x_n+y_n)+z_n)]=[(x_n+(y_n+z_n))]=X+(Y+Z).$$
Likewise,
$$X(Y+Z)=[(x_n(y_n+z_n))]=[(x_ny_n+x_nz_n)]=XY+XZ.$$
Associativity of multiplication and commutativity of both operations follow in the same way. Moreover,
$X+0_E=X$,
$X+(-X)=0_E$,
and
$X1_E=X$.
Thus $E$ is a commutative ring with identity.

The identity and zero elements are distinct. Since $0\neq 1$ in $K$, the equality set of the constant zero and constant one sequences is empty. As $\varnothing\notin\mathcal{U}$, these sequences are not equivalent. Therefore $0_E\neq 1_E$.

We now prove that every nonzero element of $E$ is invertible. Let $X=[(x_n)]\neq 0_E$. By the definition of the equivalence relation,
$X=0_E\iff\{n:x_n=0\}\in\mathcal{U}$.
Since $X\neq 0_E$, the set $\{n:x_n=0\}$ does not belong to $\mathcal{U}$. The ultrafilter property therefore implies that its complement
$A=\{n:x_n\neq 0\}$
belongs to $\mathcal{U}$.

Define a sequence $(y_n)$ by
$y_n=\begin{cases}x_n^{-1},&n\in A,\\0,&n\notin A.\end{cases}$
The inverse $x_n^{-1}$ exists whenever $n\in A$, because $K$ is a field. Put $Y=[(y_n)]$. For every $n\in A$, we have $x_ny_n=1$. Hence
$A\subseteq\{n:x_ny_n=1\}$.
By upward closure, the latter set belongs to $\mathcal{U}$. Thus $(x_ny_n)\sim(1,1,\ldots)$, and consequently
$XY=[(x_ny_n)]=1_E$.
Since multiplication is commutative, we also have $YX=1_E$. Therefore every nonzero element of $E$ has a multiplicative inverse, and $E$ is a field.

Consider the map
$j:K\to E,\qquad j(a)=[(a,a,\ldots)]$.
Coordinatewise addition and multiplication give
$j(a+b)=j(a)+j(b)$
and
$j(ab)=j(a)j(b)$
for every $a,b\in K$. Also, $j(0)=0_E$ and $j(1)=1_E$, so $j$ is a unital ring homomorphism.

To prove injectivity, let $a,b\in K$. If $a=b$, the equality set of their constant sequences is $\mathbb{N}$, so $j(a)=j(b)$. If $a\neq b$, that equality set is empty, so $j(a)\neq j(b)$. Hence
$j(a)=j(b)\iff a=b$.
Thus $j$ is injective. Since $K$ and $E$ are fields, it follows that $j$ is a field embedding. Explicitly, for $a\neq 0$, the identity
$j(a)j(a^{-1})=j(1)=1_E$
shows that $j(a^{-1})=j(a)^{-1}$.

The field $E$ has characteristic zero. Indeed, for every positive integer $m$, we have
$m1_E=j(m1_K)$.
Since $K$ has characteristic zero, $m1_K\neq 0$. Injectivity of $j$ therefore gives $m1_E\neq 0$. Thus no positive integer annihilates $1_E$, which proves $\operatorname{char}(E)=0$.

It remains to prove that $E$ is algebraically closed. Let
$$f(T)=a_dT^d+a_{d-1}T^{d-1}+\cdots+a_0\in E[T]$$
be a nonconstant polynomial, where $d\geq 1$ and $a_d\neq 0_E$.

For each $r\in\{0,\ldots,d\}$, choose a representative sequence $(a_{r,n})_{n\in\mathbb{N}}$ such that $a_r=[(a_{r,n})]$. Since the leading coefficient $a_d$ is nonzero, the same argument used above for inverses gives
$A_d=\{n\in\mathbb{N}:a_{d,n}\neq 0\}\in\mathcal{U}$.

For every $n\in\mathbb{N}$, define the coordinate polynomial
$$f_n(T)=a_{d,n}T^d+a_{d-1,n}T^{d-1}+\cdots+a_{0,n}\in K[T].$$
If $n\in A_d$, then $a_{d,n}\neq 0$, so $f_n$ has degree exactly $d$. In particular, it is nonconstant. Since $K$ is algebraically closed, $f_n$ has a root in $K$.

For each $n\in A_d$, choose $r_n\in K$ such that $f_n(r_n)=0$. These choices can be made explicitly using the countability of $K$: fix an enumeration $K=\{c_1,c_2,\ldots\}$ and choose $r_n$ to be the first element of this enumeration that is a root of $f_n$. For $n\notin A_d$, put $r_n=0$. This defines a sequence $(r_n)\in K^{\mathbb{N}}$. Set $R=[(r_n)]\in E$.

Because addition and multiplication in $E$ are defined coordinatewise, polynomial evaluation satisfies
$f(R)=[(f_n(r_n))]$.
For every $n\in A_d$, the choice of $r_n$ gives $f_n(r_n)=0$. Therefore
$A_d\subseteq\{n:f_n(r_n)=0\}$.
Since $A_d\in\mathcal{U}$ and $\mathcal{U}$ is upward closed, the zero set $\{n:f_n(r_n)=0\}$ belongs to $\mathcal{U}$. By the definition of the quotient relation, this means
$[(f_n(r_n))]=0_E$.
Consequently, $f(R)=0_E$.

We have proved that every nonconstant polynomial over $E$ has a root in $E$. Since $E$ is a field, the factor theorem then allows us to remove a linear factor and repeat the argument. Induction on the degree shows that every nonconstant polynomial over $E$ splits into linear factors over $E$. Hence $E$ is algebraically closed.
\end{proof}

\begin{proposition}
{\rm Define
$$
[(x_n)]\leq_E[(y_n)]
\quad\Longleftrightarrow\quad
\{n:x_n\leq_Ky_n\}\in\mathcal{U}.
$$
This is a compatible lattice order on $E$, with
$$
\begin{aligned}
[(x_n)]\vee_E[(y_n)]&=[(x_n\vee_Ky_n)],\\
[(x_n)]\wedge_E[(y_n)]&=[(x_n\wedge_Ky_n)].
\end{aligned}
$$
The diagonal embedding $j$ preserves and reflects order, and preserves joins and meets.}
\end{proposition}

\begin{proof}
We use the field structure on $E=K^{\mathbb{N}}/\mathcal{U}$ established in Proposition~7.2 and the compatible lattice order $\leq_K$ on $K$ established in Proposition~6.1. Write $0_E=[(0,0,\ldots)]$ and $1_E=[(1,1,\ldots)]$.

We first record a consequence of the filter properties of $\mathcal{U}$. Suppose that $C,D\subseteq\mathbb{N}$ agree on a set $H\in\mathcal{U}$, meaning that $C\cap H=D\cap H$. If $C\in\mathcal{U}$, then $C\cap H\in\mathcal{U}$ by closure under finite intersections. Since $C\cap H=D\cap H\subseteq D$, upward closure gives $D\in\mathcal{U}$. Interchanging $C$ and $D$ proves the converse. Therefore
$C\cap H=D\cap H,\ H\in\mathcal{U}\implies(C\in\mathcal{U}\iff D\in\mathcal{U})$.

We now verify that the proposed order is independent of representatives. Suppose that $(x_n)\sim(x'_n)$ and $(y_n)\sim(y'_n)$. Then
$H_x=\{n:x_n=x'_n\}\in\mathcal{U}$
and
$H_y=\{n:y_n=y'_n\}\in\mathcal{U}$.
Consequently, $H=H_x\cap H_y\in\mathcal{U}$. For every $n\in H$, we have
$x_n\leq_K y_n\iff x'_n\leq_K y'_n$.
Thus the sets
$C=\{n:x_n\leq_K y_n\}$
and
$D=\{n:x'_n\leq_K y'_n\}$
agree on $H$. By the observation above, $C\in\mathcal{U}$ if and only if $D\in\mathcal{U}$. Hence
$[(x_n)]\leq_E[(y_n)]\iff\{n:x_n\leq_K y_n\}\in\mathcal{U}$
defines a relation on equivalence classes.

The coordinatewise candidates for joins and meets are also independent of representatives. Indeed, for every $n\in H$, equality of the inputs gives
$x_n\vee_K y_n=x'_n\vee_K y'_n$
and
$x_n\wedge_K y_n=x'_n\wedge_K y'_n$.
The equality sets of these sequences therefore contain $H$ and belong to $\mathcal{U}$. Consequently,
$[(x_n\vee_K y_n)]=[(x'_n\vee_K y'_n)]$
and
$[(x_n\wedge_K y_n)]=[(x'_n\wedge_K y'_n)]$.

We next prove that $\leq_E$ is a partial order. Let $X=[(x_n)]$, $Y=[(y_n)]$, and $Z=[(z_n)]$.

For every $n$, reflexivity in $K$ gives $x_n\leq_K x_n$. Hence
$\{n:x_n\leq_K x_n\}=\mathbb{N}\in\mathcal{U}$,
so $X\leq_E X$.

Suppose that $X\leq_E Y$ and $Y\leq_E X$. The sets
$A=\{n:x_n\leq_K y_n\}$
and
$B=\{n:y_n\leq_K x_n\}$
belong to $\mathcal{U}$, and therefore $A\cap B\in\mathcal{U}$. For every $n\in A\cap B$, antisymmetry of $\leq_K$ gives $x_n=y_n$. Thus
$A\cap B\subseteq\{n:x_n=y_n\}$.
Upward closure implies $\{n:x_n=y_n\}\in\mathcal{U}$, so $(x_n)\sim(y_n)$ and $X=Y$. This proves antisymmetry.

Suppose that $X\leq_E Y$ and $Y\leq_E Z$. Then
$A=\{n:x_n\leq_K y_n\}\in\mathcal{U}$
and
$B=\{n:y_n\leq_K z_n\}\in\mathcal{U}$.
For every $n\in A\cap B$, transitivity in $K$ yields $x_n\leq_K z_n$. Hence
$A\cap B\subseteq\{n:x_n\leq_K z_n\}$.
Since $A\cap B\in\mathcal{U}$, the set on the right belongs to $\mathcal{U}$. Therefore $X\leq_E Z$, proving transitivity. Thus $\leq_E$ is a partial order.

We verify compatibility with addition. Let $A=[(a_n)]\in E$. Translation invariance of $\leq_K$ gives, for every $n$,
$x_n\leq_K y_n\iff x_n+a_n\leq_K y_n+a_n$.
Thus the two comparison sets are equal. Using the coordinatewise addition on $E$, we obtain
$$X\leq_E Y\iff X+A\leq_E Y+A.$$
In particular, translations preserve the order.

We now verify compatibility with multiplication. Suppose that $0_E\leq_E X$ and $0_E\leq_E Y$. Then
$A=\{n:0\leq_K x_n\}\in\mathcal{U}$
and
$B=\{n:0\leq_K y_n\}\in\mathcal{U}$.
Their intersection belongs to $\mathcal{U}$. For every $n\in A\cap B$, compatibility of the order on $K$ with multiplication gives $0\leq_K x_ny_n$. Therefore
$A\cap B\subseteq\{n:0\leq_K x_ny_n\}$.
By upward closure, the latter set belongs to $\mathcal{U}$, proving $0_E\leq_E XY$.

More generally, suppose that $X\leq_E Y$ and $0_E\leq_E A$. Translation invariance implies $0_E\leq_E Y-X$. The product property just proved gives
$$0_E\leq_E A(Y-X)=AY-AX.$$
Hence $AX\leq_E AY$. Since $E$ is commutative, multiplication on the right by a nonnegative element also preserves order.

Moreover, $0\leq_K 1$ at every coordinate, so
$\{n:0\leq_K 1\}=\mathbb{N}\in\mathcal{U}$.
Thus $0_E\leq_E 1_E$. Proposition~7.2 gives $0_E\neq 1_E$, and consequently $0_E<_E1_E$.

We next prove that every pair of elements has a join and a meet. Fix $X=[(x_n)]$ and $Y=[(y_n)]$, and set
$J=[(x_n\vee_K y_n)]$.
We have already proved that this class is independent of the chosen representatives.

For every $n$, the defining property of the join in $K$ gives
$x_n\leq_K x_n\vee_K y_n$
and
$y_n\leq_K x_n\vee_K y_n$.
Both comparison sets are therefore equal to $\mathbb{N}$. Hence $X\leq_E J$ and $Y\leq_E J$, so $J$ is an upper bound of $X$ and $Y$.

Let $Z=[(z_n)]$ be any other upper bound of $X$ and $Y$. Then
$A=\{n:x_n\leq_K z_n\}\in\mathcal{U}$
and
$B=\{n:y_n\leq_K z_n\}\in\mathcal{U}$.
For every $n\in A\cap B$, the element $z_n$ is an upper bound of $x_n$ and $y_n$ in $K$. Since $x_n\vee_K y_n$ is their least upper bound, we have
$x_n\vee_K y_n\leq_K z_n$.
Thus
$$A\cap B\subseteq\{n:x_n\vee_K y_n\leq_K z_n\}.$$
The comparison set on the right belongs to $\mathcal{U}$, so $J\leq_E Z$. Therefore $J$ is the least upper bound of $X$ and $Y$, proving
$X\vee_E Y=[(x_n\vee_K y_n)]$.

Similarly, set
$M=[(x_n\wedge_K y_n)]$.
For every $n$, we have
$x_n\wedge_K y_n\leq_K x_n$
and
$x_n\wedge_K y_n\leq_K y_n$.
It follows that $M\leq_E X$ and $M\leq_E Y$, so $M$ is a lower bound.

If $Z=[(z_n)]$ is any other lower bound, then
$A=\{n:z_n\leq_K x_n\}\in\mathcal{U}$
and
$B=\{n:z_n\leq_K y_n\}\in\mathcal{U}$.
For every $n\in A\cap B$, the greatest-lower-bound property in $K$ yields
$z_n\leq_K x_n\wedge_K y_n$.
Hence
$A\cap B\subseteq\{n:z_n\leq_K x_n\wedge_K y_n\}$.
The latter set belongs to $\mathcal{U}$, and therefore $Z\leq_E M$. Thus $M$ is the greatest lower bound of $X$ and $Y$, proving
$X\wedge_E Y=[(x_n\wedge_K y_n)]$.

We have therefore obtained a lattice order on $E$ that is compatible with its addition and multiplication. In particular, the underlying ring of the field $E$ is a lattice-ordered ring.

Finally, consider the diagonal field embedding $j:K\to E$ from Proposition~7.2, defined by
$j(a)=[(a,a,\ldots)]$.
For $a,b\in K$, the comparison set of the two constant sequences is $\mathbb{N}$ if $a\leq_K b$, and it is empty if $a\not\leq_K b$. Since $\mathbb{N}\in\mathcal{U}$ and $\varnothing\notin\mathcal{U}$, we obtain
$j(a)\leq_E j(b)\iff a\leq_K b$.
Thus $j$ both preserves and reflects order.

Applying the coordinate formula for joins to constant sequences gives
$$j(a)\vee_E j(b)=[(a\vee_K b,a\vee_K b,\ldots)]=j(a\vee_K b).$$
Likewise, the coordinate formula for meets gives
$$j(a)\wedge_E j(b)=[(a\wedge_K b,a\wedge_K b,\ldots)]=j(a\wedge_K b).$$
Hence $j$ preserves joins and meets.

\end{proof}

\begin{proposition}
{\rm The field $E$ has cardinality
$$
|E|=\mathfrak{c}=2^{\aleph_0}.
$$}
\end{proposition}

\begin{proof}
We prove the two cardinal inequalities separately.

Since $K$ is countable, choose an injective map $\iota:K\to\mathbb{N}$. To each sequence $x=(x_n)_{n\geq 1}\in K^{\mathbb{N}}$, associate the set
$G_x=\{(n,\iota(x_n)):n\in\mathbb{N}\}\subseteq\mathbb{N}\times\mathbb{N}$.
The assignment $x\mapsto G_x$ is injective. Indeed, if $G_x=G_y$, then for every $n$ the unique pair in this set having first coordinate $n$ has second coordinate both $\iota(x_n)$ and $\iota(y_n)$. Thus $\iota(x_n)=\iota(y_n)$, and injectivity of $\iota$ gives $x_n=y_n$ for every $n$.

Consequently,
$|K^{\mathbb{N}}|\leq|\mathcal{P}(\mathbb{N}\times\mathbb{N})|=2^{\aleph_0}$,
where the equality follows from the countability of $\mathbb{N}\times\mathbb{N}$. Since the quotient map $K^{\mathbb{N}}\to E$ is surjective, we obtain
$|E|\leq|K^{\mathbb{N}}|\leq 2^{\aleph_0}$.

For the reverse inequality, we construct an injection from the set of binary sequences into $E$. First, recall that no finite subset of $\mathbb{N}$ belongs to $\mathcal{U}$. Indeed, if $F\subseteq\mathbb{N}$ is finite, then its complement $\mathbb{N}\setminus F$ is cofinite and hence belongs to $\mathcal{U}$. If $F$ also belonged to $\mathcal{U}$, closure under finite intersections would imply
$\varnothing=F\cap(\mathbb{N}\setminus F)\in\mathcal{U}$,
contradicting the fact that $\mathcal{U}$ is a proper filter.

Because $K$ has characteristic zero, the canonical map $\mathbb{Z}\to K$ is injective. We therefore identify integers with their images in $K$. In particular, a nonzero integer remains nonzero when regarded as an element of $K$.

Let $$\varepsilon=(\varepsilon_j)_{j\geq 1}\in\{0,1\}^{\mathbb{N}}.$$ For every $n\geq 1$, define
$$a_n(\varepsilon)=\sum_{j=1}^{n}\varepsilon_j2^{j-1}\in\mathbb{Z}\subseteq K.$$
Thus $a_n(\varepsilon)$ encodes the first $n$ entries of $\varepsilon$ as an integer. Define
$\Phi:\{0,1\}^{\mathbb{N}}\to E$
by
$\Phi(\varepsilon)=[(a_n(\varepsilon))_{n\geq 1}]$.

We claim that $\Phi$ is injective. Let $\varepsilon,\eta\in\{0,1\}^{\mathbb{N}}$ be distinct, and let
$$r=\min\{j\geq 1:\varepsilon_j\neq\eta_j\}.$$
Such an index exists because the sequences are distinct and $\mathbb{N}$ is well ordered.

For every $n<r$, the first $n$ entries of the two sequences agree. Hence $a_n(\varepsilon)=a_n(\eta)$.

Now fix $n\geq r$. Computing in $\mathbb{Z}$, we have
$$a_n(\varepsilon)-a_n(\eta)=(\varepsilon_r-\eta_r)2^{r-1}+\sum_{j=r+1}^{n}(\varepsilon_j-\eta_j)2^{j-1}.$$
Every term in the final sum is divisible by $2^r$. More explicitly, put
$M_n=\sum_{j=r+1}^{n}(\varepsilon_j-\eta_j)2^{j-r-1}\in\mathbb{Z}$,
with the empty sum interpreted as zero when $n=r$. Then
$a_n(\varepsilon)-a_n(\eta)=2^{r-1}\bigl((\varepsilon_r-\eta_r)+2M_n\bigr)$.

Since $\varepsilon_r$ and $\eta_r$ are distinct binary digits, $\varepsilon_r-\eta_r$ is either $1$ or $-1$. Thus $(\varepsilon_r-\eta_r)+2M_n$ is an odd integer and cannot be zero. Therefore $a_n(\varepsilon)-a_n(\eta)$ is a nonzero integer. By the characteristic-zero embedding $\mathbb{Z}\hookrightarrow K$, it is also nonzero in $K$. Consequently,
$a_n(\varepsilon)\neq a_n(\eta)$
for every $n\geq r$.

It follows that the equality set of the two representing sequences is exactly
$\{n\in\mathbb{N}:a_n(\varepsilon)=a_n(\eta)\}=\{1,\ldots,r-1\}$,
where the right-hand side is empty if $r=1$. This set is finite, so it does not belong to $\mathcal{U}$. By the definition of the equivalence relation used to form $E$, the sequences $(a_n(\varepsilon))$ and $(a_n(\eta))$ are therefore not equivalent. Hence $\Phi(\varepsilon)\neq\Phi(\eta)$, proving injectivity.

Finally, the set $\{0,1\}^{\mathbb{N}}$ has cardinality $2^{\aleph_0}$. Indeed, the assignment
$\varepsilon\mapsto\{j\in\mathbb{N}:\varepsilon_j=1\}$
is a bijection from $\{0,1\}^{\mathbb{N}}$ onto $\mathcal{P}(\mathbb{N})$. The injection $\Phi$ consequently gives
$2^{\aleph_0}\leq|E|$.

Combining this with the upper bound and applying the Cantor--Schr\"oder--Bernstein theorem, we conclude that
$|E|=2^{\aleph_0}=\mathfrak{c}$.
\end{proof}

\section{A field isomorphism fixing the countable subfield}

\begin{lemma}
{\rm Let $L$ be an uncountable extension of a countable field $k_0$, with $|L|=\kappa$. Then
$$
\operatorname{trdeg}_{k_0}L=\kappa.
$$}
\end{lemma}

\begin{proof}
Let $L/k_0$ be the given field extension, where $k_0$ is countable and $|L|=\kappa>\aleph_0$. We first choose a transcendence basis and then estimate the cardinality of the resulting algebraic extension.

Consider the family $\mathcal{A}$ of all subsets of $L$ that are algebraically independent over $k_0$, ordered by inclusion. This family is nonempty because it contains the empty set. Let $\mathcal{C}$ be a chain in $\mathcal{A}$, and put $S=\bigcup_{A\in\mathcal{C}}A$.

The set $S$ is algebraically independent over $k_0$. Indeed, any polynomial relation among elements of $S$ involves only finitely many of them. Since $\mathcal{C}$ is a chain, those finitely many elements all belong to a single member of $\mathcal{C}$. Algebraic independence of that member excludes every nonzero polynomial relation among them. Thus $S\in\mathcal{A}$, and $S$ is an upper bound for the chain.

By Zorn's lemma, there exists a maximal algebraically independent subset $T\subseteq L$ over $k_0$. Set $F=k_0(T)$, the subfield of $L$ generated by $k_0$ and $T$.

We claim that $L$ is algebraic over $F$. Suppose otherwise. Then there exists $a\in L$ that is transcendental over $F$. We show that $T\cup\{a\}$ is algebraically independent over $k_0$, contradicting the maximality of $T$.

Suppose that a nonzero polynomial relation among finitely many distinct elements $t_1,\ldots,t_s\in T$ and $a$ were to hold. Such a relation has the form
$P(t_1,\ldots,t_s,a)=0$,
where $P(Y_1,\ldots,Y_s,Z)\in k_0[Y_1,\ldots,Y_s,Z]$ is nonzero. Write
$$P(Y_1,\ldots,Y_s,Z)=\sum_{j=0}^{d}P_j(Y_1,\ldots,Y_s)Z^j.$$
At least one coefficient polynomial $P_j$ is nonzero. Since $T$ is algebraically independent over $k_0$, every such nonzero coefficient polynomial satisfies
$P_j(t_1,\ldots,t_s)\neq 0$.
Therefore
$Q(Z)=\sum_{j=0}^{d}P_j(t_1,\ldots,t_s)Z^j$
is a nonzero polynomial in $F[Z]$. But the assumed relation gives $Q(a)=0$, contradicting the transcendence of $a$ over $F$. Hence $T\cup\{a\}$ is algebraically independent, and the claim follows.

Thus $T$ is a transcendence basis of $L$ over $k_0$, so
$\operatorname{trdeg}_{k_0}L=|T|$.
It remains to prove that $|T|=\kappa$.

Put $\mu=\max\{\aleph_0,|T|\}$. This is an infinite cardinal, regardless of whether $T$ is finite or infinite. We first prove that $|F|\leq\mu$.

Let $k_0[T]$ denote the subring of $L$ generated by $k_0$ and $T$. Every element of this ring is a polynomial expression involving only finitely many elements of $T$. Consequently,
$k_0[T]=\bigcup\{k_0[S]:S\subseteq T\text{ is finite}\}$.

The family of finite subsets of $T$ has cardinality at most $\mu$: each finite subset can be listed by a finite tuple of elements of $T$, and the set of all such finite tuples has cardinality at most $\mu$.

For each finite subset $S=\{s_1,\ldots,s_r\}\subseteq T$, the ring $k_0[S]$ is at most countable. Indeed, each of its elements is represented by a polynomial in finitely many variables over the countable field $k_0$. Such a polynomial is specified by a finite list of coefficients from $k_0$ and exponent vectors from $\mathbb{N}_0^r$. Both sets are at most countable, so there are at most countably many such finite lists. Therefore $|k_0[S]|\leq\aleph_0$.

Taking the union over all finite subsets $S\subseteq T$, we obtain
$|k_0[T]|\leq\mu\cdot\aleph_0=\mu$.

The ring $k_0[T]$ is an integral domain because it is a subring of the field $L$, and its fraction field is $F=k_0(T)$. Thus every element of $F$ can be written as $p/q$, where $p,q\in k_0[T]$ and $q\neq 0$. The number of such pairs is at most $\mu^2=\mu$, so
$|F|\leq\mu$.

We next count the polynomials over $F$. For every fixed $d\in\mathbb{N}_0$, a polynomial of degree at most $d$ is determined by its $d+1$ coefficients. Hence the set of such polynomials has cardinality at most
$|F|^{d+1}\leq\mu$.
Every polynomial has finite degree, so taking a countable union gives
$|F[X]|\leq\aleph_0\cdot\mu=\mu$.

For each nonzero polynomial $f\in F[X]$, let
$Z_L(f)=\{a\in L:f(a)=0\}$.
Since $L$ is a field, a nonzero polynomial of degree $d$ has at most $d$ distinct roots in $L$. In particular, every set $Z_L(f)$ is finite. A nonzero constant polynomial has no roots, which causes no exception to this statement.

Because $L$ is algebraic over $F$, every element of $L$ belongs to $Z_L(f)$ for some nonzero polynomial $f\in F[X]$. Therefore
$L=\bigcup_{f\in F[X]\setminus\{0\}}Z_L(f)$.
There are at most $\mu$ polynomials in this union, and each root set is finite. Consequently,
$\kappa=|L|\leq\mu\cdot\aleph_0=\mu=\max\{\aleph_0,|T|\}$.

On the other hand, $T\subseteq L$ gives $|T|\leq\kappa$. If $|T|\leq\aleph_0$, the preceding estimate would imply $\kappa\leq\aleph_0$, contrary to the assumption that $L$ is uncountable. Thus $|T|>\aleph_0$, and hence $\mu=|T|$. We obtain
$\kappa\leq|T|\leq\kappa$,
so $|T|=\kappa$.

Since $T$ is a transcendence basis, it follows that
$\operatorname{trdeg}_{k_0}L=|T|=\kappa$.
\end{proof}

\begin{lemma}
{\rm Let $L/D_0$ be an algebraic field extension, let $\Omega$ be algebraically closed, and let $\sigma_0:D_0\to\Omega$ be a field embedding. Then $\sigma_0$ extends to a field embedding $L\to\Omega$.}
\end{lemma}

\begin{proof}
Consider the set $\mathcal{P}$ of all pairs $(D,\sigma)$ such that $D$ is an intermediate field with $D_0\subseteq D\subseteq L$ and $\sigma:D\to\Omega$ is a field embedding satisfying $\sigma|_{D_0}=\sigma_0$. This set is nonempty because $(D_0,\sigma_0)\in\mathcal{P}$.

Order $\mathcal{P}$ by extension: define
$(D_1,\sigma_1)\preceq(D_2,\sigma_2)$
if $D_1\subseteq D_2$ and $\sigma_2|_{D_1}=\sigma_1$. This is a partial order.

We verify that every chain has an upper bound. Let $\{(D_\lambda,\sigma_\lambda):\lambda\in I\}$ be a nonempty chain in $\mathcal{P}$, and set
$D_*=\bigcup_{\lambda\in I}D_\lambda$.
The set $D_*$ is a subfield of $L$ containing $D_0$. Indeed, given $x,y\in D_*$, the chain property allows us to find a single field $D_\lambda$ containing both elements. Their sum, difference, and product then belong to $D_\lambda\subseteq D_*$. If $x\neq 0$, its inverse also belongs to any $D_\lambda$ containing $x$. Thus $D_*$ is a field.

Define $\sigma_*:D_*\to\Omega$ by setting $\sigma_*(x)=\sigma_\lambda(x)$ whenever $x\in D_\lambda$. This definition is independent of $\lambda$. To see this, suppose that $x\in D_\lambda\cap D_\mu$. Since the pairs form a chain, one extends the other, so their embeddings agree on the smaller domain and hence agree at $x$.

For $x,y\in D_*$, choose a member of the chain containing both. Since its embedding preserves addition and multiplication, so does $\sigma_*$. Moreover, $\sigma_*(1)=1$. If $\sigma_*(x)=0$, choose $D_\lambda$ containing $x$. Then $\sigma_\lambda(x)=0$, and injectivity of $\sigma_\lambda$ gives $x=0$. Thus $\sigma_*$ is a field embedding. It extends every $\sigma_\lambda$ and agrees with $\sigma_0$ on $D_0$. Therefore $(D_*,\sigma_*)$ is an upper bound for the chain. The empty chain has the upper bound $(D_0,\sigma_0)$.

By Zorn's lemma, $\mathcal{P}$ has a maximal element, which we denote by $(D,\sigma)$. We prove that $D=L$.

Suppose, to the contrary, that $D\neq L$, and choose $a\in L\setminus D$. Since $L/D_0$ is algebraic, the element $a$ satisfies a nonzero polynomial with coefficients in $D_0$. As $D_0\subseteq D$, the same polynomial shows that $a$ is algebraic over $D$.

Let $p(X)\in D[X]$ be the monic minimal polynomial of $a$ over $D$. Write
$p(X)=X^d+c_{d-1}X^{d-1}+\cdots+c_0$.
Then $p(a)=0$, and $p$ is irreducible over $D$.

The embedding $\sigma$ is an isomorphism from $D$ onto its image $\sigma(D)$. Applying it coefficientwise defines a ring isomorphism
$\sigma_*:D[X]\to\sigma(D)[X]$,
given by
$f(X)=\sum_j c_jX^j\longmapsto f^\sigma(X)=\sum_j\sigma(c_j)X^j$.
This isomorphism preserves degrees and polynomial products.

In particular,
$p^\sigma(X)=X^d+\sigma(c_{d-1})X^{d-1}+\cdots+\sigma(c_0)$
is monic and irreducible in $\sigma(D)[X]$. Indeed, a factorization of $p^\sigma$ into two polynomials of positive degree over $\sigma(D)$ would pull back under the coefficientwise inverse of $\sigma$ to such a factorization of $p$ over $D$, contradicting its irreducibility.

Since $\Omega$ is algebraically closed, the nonconstant polynomial $p^\sigma\in\Omega[X]$ has a root $b\in\Omega$. Thus $p^\sigma(b)=0$. The element $b$ is algebraic over $\sigma(D)$, and its monic minimal polynomial over $\sigma(D)$ divides $p^\sigma$. Since $p^\sigma$ is monic and irreducible over $\sigma(D)$, that minimal polynomial must be $p^\sigma$ itself.

We now construct an embedding of $D(a)$ into $\Omega$ extending $\sigma$. First, recall that the evaluation map
$D[X]\to D[a],\qquad f\mapsto f(a)$
has kernel $(p)$, because $p$ is the minimal polynomial of $a$. Hence $D[a]\cong D[X]/(p)$. Since $p$ is irreducible over the field $D$, the quotient $D[X]/(p)$ is a field. Therefore $D[a]$ is already a field, so $D[a]=D(a)$. In particular, every element of $D(a)$ can be written as $f(a)$ for some $f\in D[X]$.

Define
$\widetilde{\sigma}:D(a)\to\Omega$
by
$\widetilde{\sigma}(f(a))=f^\sigma(b)$.

We check that this map is well defined. Suppose that $f(a)=g(a)$ for $f,g\in D[X]$. Then $(f-g)(a)=0$, so $p$ divides $f-g$ in $D[X]$. Applying the coefficient isomorphism shows that $p^\sigma$ divides $f^\sigma-g^\sigma$ in $\sigma(D)[X]$. Evaluating at $b$ and using $p^\sigma(b)=0$, we obtain $f^\sigma(b)=g^\sigma(b)$. Thus the value assigned by $\widetilde{\sigma}$ does not depend on the polynomial representation of the element.

The map $\widetilde{\sigma}$ preserves addition and multiplication. Indeed, for $f,g\in D[X]$, we have
$$\widetilde{\sigma}(f(a)+g(a))=(f+g)^\sigma(b)=f^\sigma(b)+g^\sigma(b)$$
and
$$\widetilde{\sigma}(f(a)g(a))=(fg)^\sigma(b)=f^\sigma(b)g^\sigma(b).$$
Also, $\widetilde{\sigma}(1)=1$. Hence $\widetilde{\sigma}$ is a unital ring homomorphism.

It is injective. If $\widetilde{\sigma}(f(a))=0$, then $f^\sigma(b)=0$. Since $p^\sigma$ is the minimal polynomial of $b$ over $\sigma(D)$, it divides $f^\sigma$. Applying the inverse coefficient isomorphism shows that $p$ divides $f$. Consequently, $f(a)=0$. Thus the kernel of $\widetilde{\sigma}$ is zero, and $\widetilde{\sigma}$ is a field embedding.

For every $c\in D$, taking the constant polynomial $f(X)=c$ gives $\widetilde{\sigma}(c)=\sigma(c)$. Thus $\widetilde{\sigma}$ extends $\sigma$, and taking $f(X)=X$ gives $\widetilde{\sigma}(a)=b$.

Because $D\subseteq L$ and $a\in L$, the field $D(a)$ is an intermediate field between $D_0$ and $L$. Moreover, $a\notin D$ implies $D\subsetneq D(a)$. Therefore $(D(a),\widetilde{\sigma})$ is an element of $\mathcal{P}$ strictly extending $(D,\sigma)$, contradicting maximality.

Hence $D=L$. The embedding $\sigma:L\to\Omega$ associated with the maximal pair satisfies $\sigma|_{D_0}=\sigma_0$, as required.
\end{proof}

\begin{proposition}
{\rm There exists a field isomorphism
$$
\Phi:E\longrightarrow\mathbb{C}
$$
such that
$$
\Phi(j(a))=a\qquad(a\in K).
$$}
\end{proposition}

\begin{proof}
We retain the distinction between $K\subseteq\mathbb{C}$ and its diagonal image $j(K)\subseteq E$. Put $F=j(K)$. Since $j:K\to E$ is an injective field homomorphism, the map
$\eta:F\to K,\qquad \eta(j(a))=a$
is a well-defined field isomorphism.

By Proposition~7.2, the field $E$ is algebraically closed, and by Proposition~7.4 it has cardinality $\mathfrak{c}$. The field $\mathbb{C}$ is also algebraically closed and has cardinality $\mathfrak{c}$. Since $K$ is countable and $j$ is injective, both $K$ and $F=j(K)$ are countable. Applying Lemma~8.1 to the extensions $E/F$ and $\mathbb{C}/K$, we obtain
$\operatorname{trdeg}_F E=\mathfrak{c}$
and
$\operatorname{trdeg}_K\mathbb{C}=\mathfrak{c}$.

Choose a transcendence basis $T\subseteq E$ of $E$ over $F$, and a transcendence basis $S\subseteq\mathbb{C}$ of $\mathbb{C}$ over $K$. Thus $T$ is algebraically independent over $F$ and $E$ is algebraic over $F(T)$; similarly, $S$ is algebraically independent over $K$ and $\mathbb{C}$ is algebraic over $K(S)$. The equalities of transcendence degrees give $|T|=|S|=\mathfrak{c}$, so we may choose a bijection $\theta:T\to S$.

We first construct an isomorphism between the corresponding rational function fields. Let $F[T]$ denote the subring of $E$ generated by $F$ and $T$, and let $K[S]$ denote the subring of $\mathbb{C}$ generated by $K$ and $S$.

Every element of $F[T]$ can be written as $f(t_1,\ldots,t_r)$ for some finite list of distinct elements $t_1,\ldots,t_r\in T$ and some polynomial $f\in F[X_1,\ldots,X_r]$. Let $f^\eta\in K[X_1,\ldots,X_r]$ be the polynomial obtained by applying $\eta$ to every coefficient of $f$. Define
$$\psi(f(t_1,\ldots,t_r))=f^\eta(\theta(t_1),\ldots,\theta(t_r)).$$

We verify that this definition is independent of the chosen polynomial expression. Suppose that two expressions represent the same element of $F[T]$. By adjoining any missing variables, both expressions can be written using a common finite list $t_1,\ldots,t_r$ from $T$. Thus we may suppose that
$f(t_1,\ldots,t_r)=g(t_1,\ldots,t_r)$
for polynomials $f,g\in F[X_1,\ldots,X_r]$. Then $(f-g)(t_1,\ldots,t_r)=0$. Algebraic independence of $T$ over $F$ implies that $f-g$ is the zero polynomial. Applying $\eta$ to its coefficients gives $f^\eta=g^\eta$, and therefore the two prescribed images are equal. Hence $\psi$ is well defined.

Coefficient transport and substitution preserve addition and multiplication, so $\psi$ is a unital ring homomorphism. Applying the same construction with $\eta^{-1}$ and $\theta^{-1}$ gives an inverse map from $K[S]$ to $F[T]$. The algebraic independence of $S$ ensures that this inverse is well defined by the same argument. Consequently,
$\psi:F[T]\to K[S]$
is a ring isomorphism satisfying $\psi(c)=\eta(c)$ for $c\in F$ and $\psi(t)=\theta(t)$ for $t\in T$.

The fraction fields of $F[T]$ and $K[S]$ are $F(T)$ and $K(S)$, respectively. Extend $\psi$ to a map
$\phi_0:F(T)\to K(S)$
by defining
$\phi_0(p/q)=\psi(p)/\psi(q)$
for $p,q\in F[T]$ with $q\neq 0$.

The denominator is nonzero because $\psi$ is injective. To verify independence of the fraction representation, suppose that $p/q=p'/q'$ with $q,q'\neq 0$. Then $pq'=p'q$ in the integral domain $F[T]$. Applying $\psi$ gives $\psi(p)\psi(q')=\psi(p')\psi(q)$, so the resulting fractions are equal. Thus $\phi_0$ is well defined. It preserves addition and multiplication by the usual rules for fractions, and the corresponding extension of $\psi^{-1}$ is its inverse. Therefore $\phi_0$ is a field isomorphism.

In particular, $\phi_0$ extends $\eta$ on $F$ and sends every $t\in T$ to $\theta(t)$. Hence
$\phi_0(j(a))=a$
for every $a\in K$.

Since $E/F(T)$ is algebraic and $\mathbb{C}$ is algebraically closed, Lemma~8.2 applies to the embedding $\phi_0:F(T)\to\mathbb{C}$. It gives a field embedding
$\Phi:E\to\mathbb{C}$
such that $\Phi|_{F(T)}=\phi_0$.

It remains to prove that $\Phi$ is surjective. Let $M=\Phi(E)$, which is a subfield of $\mathbb{C}$. Since $\Phi$ extends $\phi_0$ and $\phi_0(F(T))=K(S)$, we have
$K(S)\subseteq M\subseteq\mathbb{C}$.

We first show that $M$ is algebraically closed. Let
$h(X)=c_dX^d+\cdots+c_0\in M[X]$
be nonconstant, with $d\geq 1$ and $c_d\neq 0$. For each coefficient, choose its preimage $a_r\in E$ under $\Phi$, so $\Phi(a_r)=c_r$. Since $\Phi$ is injective, $a_d\neq 0$. Thus
$g(X)=a_dX^d+\cdots+a_0\in E[X]$
is nonconstant.

Because $E$ is algebraically closed, there exists $u\in E$ such that $g(u)=0$. Applying $\Phi$ yields
$h(\Phi(u))=\Phi(g(u))=0$.
The element $\Phi(u)$ belongs to $M$, so every nonconstant polynomial over $M$ has a root in $M$. Hence $M$ is algebraically closed.

Next, $\mathbb{C}$ is algebraic over $M$. Indeed, $\mathbb{C}$ is algebraic over $K(S)$ because $S$ is a transcendence basis. For any $z\in\mathbb{C}$, there is therefore a nonzero polynomial $f\in K(S)[X]$ with $f(z)=0$. Since $K(S)\subseteq M$, the same polynomial belongs to $M[X]$ and remains nonzero. Thus $z$ is algebraic over $M$.

An algebraically closed field has no proper algebraic field extension. To apply this fact explicitly, fix $z\in\mathbb{C}$ and let $m_z(X)\in M[X]$ be its monic minimal polynomial over $M$. This polynomial is irreducible and has positive degree. Since $M$ is algebraically closed, it has a root $w\in M$. The factor theorem gives $X-w$ as a factor of $m_z$. Irreducibility therefore forces $m_z(X)=X-w$. Evaluating at $z$ gives $z=w\in M$.

Thus every element of $\mathbb{C}$ belongs to $M$, so $\Phi(E)=M=\mathbb{C}$. Therefore $\Phi$ is a field isomorphism.

Finally, for every $a\in K$, we have $j(a)\in F\subseteq F(T)$, and hence
$\Phi(j(a))=\phi_0(j(a))=\eta(j(a))=a$.
Thus the isomorphism has the required restriction to the diagonal copy of $K$.
\end{proof}

\section{The lattice order on the usual complex field}

\begin{theorem}
{\rm Let $t\in\mathbb{R}$ satisfy $t>1$ and be transcendental over $\mathbb{Q}$. There exists a partial order $\preceq$ on the usual complex field such that
$$
\begin{aligned}
z\preceq w&\ \Longrightarrow\ z+a\preceq w+a
&& (z,w,a\in\mathbb{C}),\\
0\preceq z,\ 0\preceq w&\ \Longrightarrow\ 0\preceq zw
&& (z,w\in\mathbb{C}),
\end{aligned}
$$
and every pair has a least upper bound and a greatest lower bound. The order can be chosen so that
$$
0\prec1,\ 0\vee i=t\text{ and }1\wedge t=0.
$$
It restricts to the usual order on $\mathbb{Q}$ and to the order $\leq_K$ on the subfield $K$ constructed above.}
\end{theorem}

\begin{proof}
Fix a real number $t>1$ that is transcendental over $\mathbb{Q}$. Let $K\subseteq\mathbb{C}$ be the countable algebraically closed subfield constructed above, and choose the basis $B$ as in Theorem~5.1 and Proposition~6.2. In particular, $K$ contains $\mathbb{Q}$, $t$, and $i$.

Put $P_K=P_B$, and let $\leq_K$ be the compatible lattice order supplied by Proposition~6.1. Its restriction to $\mathbb{Q}$ is the usual order, and Propositions~6.1 and~6.2 give
$0<_K1$,
$0\vee_K i=t$,
and
$1\wedge_K t=0$.

Choose the ultrafilter $\mathcal{U}$ from Lemma~7.1, and let $E=K^{\mathbb{N}}/\mathcal{U}$. By Proposition~7.3, the field $E$ carries a compatible lattice order $\leq_E$, and the diagonal embedding
$j:K\to E,\qquad j(a)=[(a,a,\ldots)]$
preserves and reflects order and preserves joins and meets.

By Proposition~8.3, there is a field isomorphism $\Phi:E\to\mathbb{C}$ satisfying $\Phi(j(a))=a$ for every $a\in K$. Since $\Phi$ is bijective, this also gives
$\Phi^{-1}(a)=j(a)$
for every $a\in K$.

Let $P_E=\{X\in E:0_E\leq_E X\}$ be the nonnegative cone of $E$. Explicitly,
$P_E=\{[(x_n)]\in E:\{n:x_n\in P_K\}\in\mathcal{U}\}$.
This description is independent of representatives by Proposition~7.3.

Define
$P_{\mathbb{C}}=\Phi(P_E)$
and define a relation $\preceq$ on $\mathbb{C}$ by
$z\preceq w\iff w-z\in P_{\mathbb{C}}$.
All additions and multiplications on $\mathbb{C}$ appearing below are its usual field operations.

The relation can equivalently be described by transporting the order on $E$ through $\Phi$. Indeed, because $\Phi^{-1}$ is a field homomorphism,
$\Phi^{-1}(w-z)=\Phi^{-1}(w)-\Phi^{-1}(z)$.
Consequently,
$z\preceq w\iff\Phi^{-1}(w-z)\in P_E\iff\Phi^{-1}(z)\leq_E\Phi^{-1}(w)$.

For completeness, we verify directly the cone properties and the resulting order properties. Since $\leq_E$ is a compatible partial order, its nonnegative cone satisfies
$0_E\in P_E$,
$P_E+P_E\subseteq P_E$,
$P_E\cap(-P_E)=\{0_E\}$,
and
$P_EP_E\subseteq P_E$.

Since $\Phi(0_E)=0$, we have $0\in P_{\mathbb{C}}$. If $z,w\in P_{\mathbb{C}}$, there exist $X,Y\in P_E$ such that $z=\Phi(X)$ and $w=\Phi(Y)$. As $X+Y\in P_E$ and $\Phi$ preserves addition, we obtain
$z+w=\Phi(X)+\Phi(Y)=\Phi(X+Y)\in P_{\mathbb{C}}$.
Thus $P_{\mathbb{C}}+P_{\mathbb{C}}\subseteq P_{\mathbb{C}}$.

If both $z$ and $-z$ belong to $P_{\mathbb{C}}$, then $X=\Phi^{-1}(z)$ and $-X=\Phi^{-1}(-z)$ both belong to $P_E$. Hence $X=0_E$, so $z=0$. Therefore
$P_{\mathbb{C}}\cap(-P_{\mathbb{C}})=\{0\}$.

If $z,w\in P_{\mathbb{C}}$, write again $z=\Phi(X)$ and $w=\Phi(Y)$ with $X,Y\in P_E$. Since $XY\in P_E$ and $\Phi$ preserves multiplication,
$zw=\Phi(X)\Phi(Y)=\Phi(XY)\in P_{\mathbb{C}}$.
Thus $P_{\mathbb{C}}P_{\mathbb{C}}\subseteq P_{\mathbb{C}}$.

We now check that $\preceq$ is a partial order. For every $z\in\mathbb{C}$, we have $z-z=0\in P_{\mathbb{C}}$, so $z\preceq z$. If $z\preceq w$ and $w\preceq z$, then both $w-z$ and $z-w=-(w-z)$ belong to $P_{\mathbb{C}}$. The intersection property implies $w-z=0$, so $z=w$. Finally, if $z\preceq w$ and $w\preceq u$, then
$u-z=(u-w)+(w-z)\in P_{\mathbb{C}}$,
and hence $z\preceq u$. This proves reflexivity, antisymmetry, and transitivity.

The order is translation invariant. For any $z,w,a\in\mathbb{C}$, the identity
$(w+a)-(z+a)=w-z$
gives
$z\preceq w\iff z+a\preceq w+a$.

It is also compatible with multiplication. If $0\preceq z$ and $0\preceq w$, then $z,w\in P_{\mathbb{C}}$, so the multiplicative closure of this cone gives $zw\in P_{\mathbb{C}}$. Equivalently,
$0\preceq z,\ 0\preceq w\implies 0\preceq zw$.
More generally, if $z\preceq w$ and $0\preceq a$, then
$aw-az=a(w-z)\in P_{\mathbb{C}}$,
so $az\preceq aw$.

We next prove that every pair of complex numbers has a least upper bound and a greatest lower bound. Let $z,w\in\mathbb{C}$, and put $X=\Phi^{-1}(z)$ and $Y=\Phi^{-1}(w)$. Define
$u=\Phi(X\vee_E Y)$
and
$l=\Phi(X\wedge_E Y)$.

Since $X\leq_E X\vee_E Y$ and $Y\leq_E X\vee_E Y$, the order-transport equivalence established above gives $z\preceq u$ and $w\preceq u$. Thus $u$ is an upper bound of $z$ and $w$.

Suppose that $v\in\mathbb{C}$ is any other upper bound, so $z\preceq v$ and $w\preceq v$. Transporting these inequalities to $E$ gives
$X\leq_E\Phi^{-1}(v)$
and
$Y\leq_E\Phi^{-1}(v)$.
By the least-upper-bound property in $E$,
$X\vee_E Y\leq_E\Phi^{-1}(v)$.
Transporting this inequality back to $\mathbb{C}$ yields $u\preceq v$. Therefore $u$ is the least upper bound, and
$z\vee w=\Phi\bigl(\Phi^{-1}(z)\vee_E\Phi^{-1}(w)\bigr)$.

Similarly, $X\wedge_E Y\leq_E X$ and $X\wedge_E Y\leq_E Y$ imply $l\preceq z$ and $l\preceq w$. If $v$ is any other lower bound of $z$ and $w$, then
$\Phi^{-1}(v)\leq_E X$
and
$\Phi^{-1}(v)\leq_E Y$.
The greatest-lower-bound property in $E$ gives
$\Phi^{-1}(v)\leq_E X\wedge_E Y$,
and hence $v\preceq l$. Thus $l$ is the greatest lower bound, and
$z\wedge w=\Phi\bigl(\Phi^{-1}(z)\wedge_E\Phi^{-1}(w)\bigr)$.
Consequently, $\preceq$ is a compatible lattice order on the usual complex field.

It remains to verify its restriction to $K$ and the specified identities. Let $a,b\in K$. Since $\Phi^{-1}(a)=j(a)$ and $\Phi^{-1}(b)=j(b)$, we have
$a\preceq b\iff j(a)\leq_E j(b)\iff a\leq_K b$,
where the last equivalence uses the fact that $j$ preserves and reflects order. Thus the restriction of $\preceq$ to $K$ is precisely $\leq_K$.

The lattice operations on $K$ are preserved as well. Since $j$ preserves joins,
$$a\vee b=\Phi\bigl(j(a)\vee_E j(b)\bigr)=\Phi\bigl(j(a\vee_K b)\bigr)=a\vee_K b.$$
Likewise, preservation of meets gives
$$a\wedge b=\Phi\bigl(j(a)\wedge_E j(b)\bigr)=\Phi\bigl(j(a\wedge_K b)\bigr)=a\wedge_K b.$$

In particular, for $p,q\in\mathbb{Q}\subseteq K$, we obtain
$p\preceq q\iff p\leq_K q\iff p\leq q$,
where the final inequality is the usual rational order. Therefore $\preceq$ restricts to the usual order on $\mathbb{Q}$.

Write $z\prec w$ when $z\preceq w$ and $z\neq w$. Since $0<_K1$ and the order on $K$ is preserved, we have $0\preceq 1$. As $0\neq 1$ in $\mathbb{C}$, this gives $0\prec 1$. Finally, the identities from Proposition~6.2 transfer without change:
$0\vee i=0\vee_K i=t$
and
$1\wedge t=1\wedge_K t=0$.
\end{proof}

\begin{proposition}
Fix the choices made in the construction. Given $z,w\in\mathbb{C}$, choose representatives
$$
\Phi^{-1}(z)=[(x_n)],\
\Phi^{-1}(w)=[(y_n)],
$$
and write their unique basis expansions
$$
x_n=\sum_{b\in B}x_{n,b}b,\
y_n=\sum_{b\in B}y_{n,b}b, x_{n,b},y_{n,b}\in\mathbb{Q}.
$$
Then
$$
z\preceq w
\quad\Longleftrightarrow\quad
\bigl\{n\mid\forall b\in B\ x_{n,b}\leq y_{n,b}\bigr\}\in\mathcal{U}.
$$
Moreover,
$$
z\vee w
=\Phi[(
\sum_{b\in B}\max\{x_{n,b},y_{n,b}\}\,b
)_{n\geq1}]
$$
and
$$
z\wedge w
=\Phi[(\sum_{b\in B}\min\{x_{n,b},y_{n,b}\}\,b
)_{n\geq1}]$$
Every sum is finite for each fixed $n$; no uniform finite support across all $n$ is required.
\end{proposition}

\begin{proof}
Fix the basis $B$, the ultrafilter $\mathcal{U}$, and the field isomorphism $\Phi:E\to\mathbb{C}$ chosen in the preceding construction. Let $z,w\in\mathbb{C}$, and choose representatives such that
$X=\Phi^{-1}(z)=[(x_n)]$
and
$Y=\Phi^{-1}(w)=[(y_n)]$.

For every $n\geq 1$, the basis property of $B$ gives unique expansions
$x_n=\sum_{b\in B}x_{n,b}b$
and
$y_n=\sum_{b\in B}y_{n,b}b$,
where $x_{n,b},y_{n,b}\in\mathbb{Q}$ and only finitely many coefficients are nonzero. All comparisons, maxima, and minima of these coefficients are taken in the usual order on $\mathbb{Q}$.

For each $n$, put
$S_n=\operatorname{supp}_B(x_n)\cup\operatorname{supp}_B(y_n)$.
This is a finite subset of $B$. If $b\notin S_n$, then $x_{n,b}=y_{n,b}=0$.

We first establish the comparison formula. By Proposition~6.1, for every fixed $n$,
$x_n\leq_K y_n\iff x_{n,b}\leq y_{n,b}\text{ for every }b\in B$.
Since both coefficients vanish outside $S_n$, the condition on the right can also be checked using only the finite set $S_n$.

It follows that the two subsets of $\mathbb{N}$
$\{n:x_n\leq_K y_n\}$
and
$C=\{n:(\forall b\in B)\ x_{n,b}\leq y_{n,b}\}$
are equal. The definition of the transported order on $\mathbb{C}$ and Proposition~7.3 therefore give
$$z\preceq w\iff X\leq_E Y\iff\{n:x_n\leq_K y_n\}\in\mathcal{U}\iff C\in\mathcal{U}.$$
This is the asserted comparison formula. The universal condition over $B$ is evaluated at each fixed index $n$ when defining the set $C$.

We next prove the formulas for joins and meets. For each $n\geq 1$, define
$u_n=\sum_{b\in B}\max\{x_{n,b},y_{n,b}\}b$
and
$v_n=\sum_{b\in B}\min\{x_{n,b},y_{n,b}\}b$.
These expressions are well defined as elements of $K$. Indeed, if $b\notin S_n$, then both coefficients are zero, so their maximum and minimum are also zero. Thus
$\operatorname{supp}_B(u_n)\subseteq S_n$
and
$\operatorname{supp}_B(v_n)\subseteq S_n$.
Both sums are consequently finite.

By the coordinate formulas in Proposition~6.1, we have
$u_n=x_n\vee_K y_n$
and
$v_n=x_n\wedge_K y_n$
for every $n$. Applying Proposition~7.3 to the representative sequences gives
$X\vee_E Y=[(x_n\vee_K y_n)]=[(u_n)]$
and
$X\wedge_E Y=[(x_n\wedge_K y_n)]=[(v_n)]$.

The lattice operations on $\mathbb{C}$ were defined by transport through $\Phi$. Hence
$$z\vee w=\Phi(X\vee_E Y)=\Phi([(u_n)]),$$
which yields
$$z\vee w=\Phi\left(\left[\left(\sum_{b\in B}\max\{x_{n,b},y_{n,b}\}b\right)_{n\geq 1}\right]\right).$$
Likewise,
$$z\wedge w=\Phi(X\wedge_E Y)=\Phi([(v_n)]),$$
and therefore
$$z\wedge w=\Phi\left(\left[\left(\sum_{b\in B}\min\{x_{n,b},y_{n,b}\}b\right)_{n\geq 1}\right]\right).$$

We finally verify explicitly that all three formulas are independent of the representative sequences. Suppose that $(x'_n)$ and $(y'_n)$ are alternative representatives of $X$ and $Y$, respectively. Then
$H_x=\{n:x_n=x'_n\}\in\mathcal{U}$
and
$H_y=\{n:y_n=y'_n\}\in\mathcal{U}$.
Thus $H=H_x\cap H_y\in\mathcal{U}$.

Write the unique basis expansions of the new representatives as
$x'_n=\sum_{b\in B}x'_{n,b}b$
and
$y'_n=\sum_{b\in B}y'_{n,b}b$.
For every $n\in H$, the equalities $x_n=x'_n$ and $y_n=y'_n$, together with uniqueness of basis coordinates, imply
$x_{n,b}=x'_{n,b}$
and
$y_{n,b}=y'_{n,b}$
for every $b\in B$ simultaneously.

Let
$C'=\{n:(\forall b\in B)\ x'_{n,b}\leq y'_{n,b}\}$.
The preceding coordinate equalities show that $C\cap H=C'\cap H$. If $C\in\mathcal{U}$, then $C\cap H\in\mathcal{U}$, and the inclusion $C\cap H\subseteq C'$ implies $C'\in\mathcal{U}$. The same argument with $C$ and $C'$ interchanged proves the converse. Therefore
$C\in\mathcal{U}\iff C'\in\mathcal{U}$,
so the comparison formula is independent of the representatives.

Define $u'_n$ and $v'_n$ from the new coefficients by the same maximum and minimum formulas. For every $n\in H$, equality of all corresponding coefficients gives $u_n=u'_n$ and $v_n=v'_n$. Hence the equality sets
$\{n:u_n=u'_n\}$
and
$\{n:v_n=v'_n\}$
both contain $H$ and therefore belong to $\mathcal{U}$. Consequently,
$[(u_n)]=[(u'_n)]$
and
$[(v_n)]=[(v'_n)]$.
Applying the fixed map $\Phi$ proves that the displayed join and meet formulas are also independent of representatives.

Throughout the argument, finiteness is required separately for each index $n$. The finite set $S_n$ may vary with $n$: each $u_n$ and $v_n$ is an element of $K$, so the resulting sequences belong to $K^{\mathbb{N}}$. No common finite support for all indices is needed.
\end{proof}

\vspace{0.2cm}
{\bf AI Declaration:} The author declares that all mathematical arguments and results presented in this manuscript are original and were developed and verified by the author. No AI-based tools or large language models were used at any stage of this work, including mathematical reasoning, writing, language editing, or typesetting. The author takes full responsibility for the content of the paper.

\end{document}